\documentclass[10pt]{article}
\usepackage{amsmath, amsthm, amssymb, amsfonts, mathrsfs, mathtools}
\usepackage{graphicx}
\usepackage{makeidx}
\usepackage[left=2cm,right=2cm,top=2cm,bottom=2cm]{geometry}
\usepackage{caption}
\usepackage{subcaption}
\usepackage[section]{placeins}
\usepackage{enumitem}
\usepackage{tikz}
\usepackage{stmaryrd}

\newcommand{\Mod}[1]{\ (\mathrm{mod}\ #1)}

\usetikzlibrary{decorations.pathreplacing}
\tikzstyle{vertex}=[auto=left,circle,draw=black,fill=white, inner sep=1.5]

\newtheorem{theorem}{Theorem}[section]

\newtheorem{lema}[theorem]{Lemma}
\newtheorem{corollary}[theorem]{Corollary}

\newtheorem{ex}[theorem]{Example}

\def \Zl {{\mathbb Z}}

\def \Rl {{\mathbb R}}

\title{HS-integral and Eisenstein integral mixed circulant graphs}

\author{Monu Kadyan and Bikash Bhattacharjya\\
Department of Mathematics\\
Indian Institute of Technology Guwahati, India\\
monu.kadyan@iitg.ac.in; b.bikash@iitg.ac.in}

\date{}

\begin{document}
	\maketitle
	
	\vspace{-0.3in}
	
\begin{center}{\textbf{Abstract}}\end{center}
	
\noindent A mixed graph is called \emph{second kind hermitian integral} (\emph{HS-integral}) if the eigenvalues of its Hermitian-adjacency matrix of the second kind are integers. A mixed graph is called \emph{Eisenstein integral} if the eigenvalues of its (0, 1)-adjacency matrix are Eisenstein integers. We characterize the set $S$ for which a mixed circulant graph $\text{Circ}(\mathbb{Z}_n, S)$ is HS-integral. We also show that a mixed circulant graph is Eisenstein integral if and only if it is HS-integral. Further, we express the eigenvalues and the  HS-eigenvalues of unitary oriented  circulant graphs in terms of generalized M$\ddot{\text{o}}$bius function.

\vspace*{0.3cm}
\noindent 
\textbf{Keywords.} integral graphs; HS-integral mixed graph; Eisenstein integer; mixed circulant graph. \\
\textbf{Mathematics Subject Classifications:} 05C50, 05C25

\section{Introduction}

A \emph{mixed graph} $G$ is a pair $(V(G),E(G))$, where $V(G)$ is a nonempty finite
set and $E(G)$ is a subset of $\left(V(G) \times V(G)\right)\setminus \{(u,u)~|~u\in V(G)\}$. The sets $V(G)$ and $E(G)$ are called the vertex set and the edge set of $G$, respectively. If $(v,u)\in E(G)$ if and only if $(u,v)\in E(G)$, then $G$ is called a \emph{simple graph}. If $(v,u)\not\in E(G)$ whenever $(u,v)\in E(G)$), then $G$ is called an \emph{oriented graph}.

Let $G$ be a mixed graph on $n$ vertices. The $(0, 1)$-\textit{adjacency matrix} and the \textit{Hermitian-adjacency matrix of the second kind} of $G$ are denoted by $\mathcal{A}(G):=(a_{uv})_{n\times n}$ and $\mathcal{H}(G):=(h_{uv})_{n\times n}$, respectively, where 
\[a_{uv} = \left\{ \begin{array}{rl}
	1 &\mbox{ if }
	(u,v)\in E \\ 
	0 &\textnormal{ otherwise,}
\end{array}\right.     ~~~~~\text{ and }~~~~~~ h_{uv} = \left\{ \begin{array}{cl}
	1 &\mbox{ if }
	(u,v)\in E \textnormal{ and } (v,u)\in E \\ \frac{1+i\sqrt{3}}{2} & \mbox{ if } (u,v)\in E \textnormal{ and } (v,u)\not\in E \\
	\frac{1-i\sqrt{3}}{2} & \mbox{ if } (u,v)\not\in E \textnormal{ and } (v,u)\in E\\
	0 &\textnormal{ otherwise.}
\end{array}\right.\] 

The Hermitian-adjacency matrix of the second kind was introduced by Bojan Mohar~\cite{mohar2020new}. The matrix obtained by replacing the numbers $\frac{1+i\sqrt{3}}{2}$ and $\frac{1-i\sqrt{3}}{2}$ of $\mathcal{H}(G)$ by the numbers $i$ and $-i$, respectively, is called the \emph{Hermitian-adjacency} matrix (of the first kind) of $G$. Indeed, these Hermitian-adjacency matrices of mixed graphs are special cases of the adjacency matrix of weighted directed graphs discussed in \cite{bapat2012weighted}.

By an \emph{eigenvalue} (resp. \emph{HS-eigenvalue}) of $G$, we mean an eigenvalue of $\mathcal{A}(G)$ (resp. $\mathcal{H}(G)$). Similarly, the spectrum (resp. \emph{HS-spectrum}) of $G$ is the multi-set of the eigenvalues (resp. HS-eigenvalues) of $G$.

A simple graph is said to be \textit{integral} if all of its eigenvalues are integers. A mixed graph $G$ is said to be \textit{HS-integral} if all of its HS-eigenvalues are integers. A mixed graph $G$ is said to be \textit{Eisenstein integral} if all of its eigenvalues are Eisenstein integers. Note that complex numbers of the form $a+b\omega_3$, where $a,b\in \mathbb{Z}, \omega_3=\frac{-1+i\sqrt{3}}{2}$, are known as \emph{Eisenstein integers}. If $G$ is a simple graph, then $\mathcal{A}(G)=\mathcal{H}(G)$. As a result, the terms HS-eigenvalue, HS-spectrum, and HS-integrality of $G$ have the same meaning with the terms eigenvalue, spectrum, and integrality of $G$ in the case of a simple graph $G$.

In 1974, Harary and Schwenk~\cite{harary1974graphs} proposed a characterization of integral  graphs. This problem has inspired a lot of interest over the last decades. For more results on integral graphs, we refer the reader to \cite{ahmadi2009graphs, balinska2002survey, csikvari2010integral, watanabe1979note, watanabe1979integral}.

We consider $\Gamma$ to be a finite group throughout the paper. Let $S$ be a subset of $\Gamma$ that does not contain the identity element.  If $S$ is closed under inverse (resp. $a^{-1} \not\in S$ for all $a\in S$), it is said to be \textit{symmetric} (resp. \textit{skew-symmetric}). Define $\overline{S}= \{u\in S: u^{-1}\not\in S \}$. Then $S\setminus \overline{S}$ is symmetric, while $\overline{S}$ is skew-symmetric. The \textit{mixed Cayley graph} $\text{Cay}(\Gamma,S)$ is a mixed graph with $V(\text{Cay}(\Gamma,S))=\Gamma$ and $E(\text{Cay}(\Gamma,S))=\{ (a,b):a,b\in \Gamma, ba^{-1}\in S \}$. If $S$ is symmetric (resp. skew-symmetric), then $\text{Cay}(\Gamma,S)$ is a \textit{simple Cayley graph} (resp. \textit{oriented Cayley graph}). If $G=\Zl_n$, then the graph $\text{Cay}(\Gamma, S)$ is called a \emph{circulant} graph, and it is denoted by $\text{Circ}(\mathbb{Z}_n, S)$. 

In this paper, we characterize the set $S$ for which a mixed circulant graph $\text{Circ}(\mathbb{Z}_n, S)$ is HS-integral. We also show that a mixed circulant graph is Eisenstein integral if and only if it is HS-integral.  Indeed, the characterizations of this paper are special cases of that in~\cite{ours1}, in which we considered the group to be abelian. However, due to the speciality of the group $\Zl_n$,  intermediary results and their proof techniques in this paper are different than that in~\cite{ours1}. In this paper, we also express the eigenvalues and the  HS-eigenvalues of unitary oriented  circulant graphs in terms of generalized M$\ddot{\text{o}}$bius function, which was not considered in~\cite{ours1}. We discussed integrality of the eigenvalues of the Hermitian-adjacency matrix (of the first kind) of mixed circulant graphs in~\cite{ours2}. The results in this paper are influenced from that in~\cite{ours2}. As a result, proof techniques and flow of results in this paper have some similarities with that in~\cite{ours2}.

This paper is organized as follows. In Section~\ref{prel}, some preliminary concepts and results are discussed. In particular, we express the HS-eigenvalues of a mixed circulant graph as a sum of HS-eigenvalues of a simple circulant graph and an oriented circulant graph. In Section~\ref{sufficient}, we obtain a sufficient condition on the connection set for the HS-integrality of an oriented circulant graph. In Section~\ref{necessity-mixed}, we first characterize HS-integrality of oriented circulant graphs by proving the necessity of the condition obtained in Section~\ref{sufficient}. After that, we extend this characterization to mixed circulant graphs. In Section~\ref{Eisenstein}, we prove that a mixed circulant graph is Eisenstein integral if and only if it is HS-integral. In the last section, we express the eigenvalues and the HS-eigenvalues of unitary oriented circulant graphs in terms of generalized M$\ddot{\text{o}}$bius function.

\section{Preliminaries}\label{prel}

An $n\times n$ \textit{circulant matrix} $C$ have the form 
	
	\begin{equation*}
		C = 
		\begin{bmatrix}
			c_{0} & c_{1} &c_2& \cdots & c_{n-1} \\
			c_{n-1} & c_{0} &c_1& \cdots & c_{n-2} \\
			c_{n-2} & c_{n-1} &c_0& \cdots & c_{n-3} \\
			\vdots & \vdots & \vdots & \ddots & \vdots \\
			c_{1} & c_{2} &c_3& \cdots & c_{0} 
		\end{bmatrix},
	\end{equation*} 
	where each row is a cyclic shift of the row above it. The circulant matrix $C$ is also denoted by $\text{Circ}(c_0,c_1,\ldots,c_{n-1})$. Note that the $(j,k)$-th entry of $C$ is $c_{k-j\Mod n}$. A circulant matrix is diagonalizable by the matrix $F$ whose $j$-th column is given by 
	$$F_j= \frac{1}{\sqrt{n}}\left[ 1 ~~ \omega_n^{j} ~~ \hdots ~~ \omega_n^{(n-1)j}\right]^T,$$
	 where $\omega_n=\exp{\left(\frac{2\pi i}{n}\right)}$ and $0\leq j \leq n-1$. The eigenvalues of $C$ are given by
\begin{equation}\label{circev}
\lambda_j=\sum\limits_{k=0}^{n-1}c_k \omega_n^{jk} \textnormal{ for } j\in \{0,1,...,n-1\}.
\end{equation} 
See \cite{circulant} for more details on circulant matrices.

\begin{lema}\label{SpecMixCayGraph} Let $S$ be a subset of $\mathbb{Z}_n$ such that $0\notin S$. Then the HS-spectrum of the mixed circulant graph $\text{Circ}(\mathbb{Z}_n,S)$ is $\{\gamma_0,\gamma_1,...,\gamma_{n-1} \}$, where $\gamma_j=\lambda_j+\mu_j$, 
$$\lambda_j= \sum\limits_{k\in S \setminus \overline{S}} \omega_n^{jk}~~ \text{ and } ~~\mu_j=  \sum\limits_{k\in \overline{S}} (\omega_6\omega_n^{jk} + \omega_6^5\omega_n^{-jk}) \textnormal{ for } j\in \{0,1,...,n-1\}.$$
\end{lema}
\begin{proof}Let
	$$c_{s} = \left\{ \begin{array}{rl}
		1 & \mbox{ if } s \in S \setminus \overline{S}\\
		\omega_6 & \mbox{ if } s \in \overline{S} \\
		\omega_6^5 & \mbox{ if } s \in \overline{S}^{-1}\\
		0 &\textnormal{ otherwise.}
	\end{array}\right.$$ 
Then the HS-spectrum of $\text{Circ}(\mathbb{Z}_n,S)$ is same as the spectrum of $\text{Circ}(c_0,...,c_{n-1})$. Now the proof follows from Equation~(\ref{circev}). 
\end{proof}

From Lemma~\ref{SpecMixCayGraph}, we observe that the HS-eigenvalues of a mixed circulant graph $\text{Circ}(\mathbb{Z}_n,S)$ are the sum of the HS-eigenvalues of the mixed graphs $\text{Circ}(\mathbb{Z}_n,S \setminus \overline{S})$ and $\text{Circ}(\mathbb{Z}_n,\overline{S})$.  Note that $\omega_6=-\omega_3^2$ and $\omega_6^5=-\omega_3$. Therefore the eigenvalue $\mu_j$ can also be written as $\mu_j= -\sum\limits_{k\in \overline{S}} (\omega_3^2\omega_n^{jk} + \omega_3\omega_n^{-jk})$. Further, one can see that the eigenvalues of $\text{Circ}(\mathbb{Z}_n,S)$ are given by $\sum\limits_{k\in S} \omega_n^{jk}$ for each $j\in \{0,1,...,n-1\}$.

Let $n\geq 2$ be a fixed positive integer. We review some basic definitions and notations from \cite{2006integral}. For a divisor $d$ of $n$, define 
\begin{align*}
&M_n(d)=\{dk: 1\leq dk\leq n-1\}~\text{ and}\\
&G_n(d)=\{dk: 1\leq dk\leq n-1, \gcd(dk,n)=d \}. 
\end{align*}
It is clear that $M_n(n)=G_n(n)=\emptyset$, $M_n(d)=dM_{\frac{n}{d}}(1)$ and $G_n(d)=dG_{\frac{n}{d}}(1)$.
	
	\begin{lema} \cite{2006integral}\label{BasicProp} If $n=dg$ for some $d,g\in \mathbb{Z}$ then $M_n(d)=\bigcup\limits_{h\mid g} G_n(hd)$.
	\end{lema}

	Wasin So \cite{2006integral} characterized integral circulant graphs in the following theorem.

	\begin{theorem}\cite{2006integral}\label{2006integral}
		The simple circulant graph $\text{Circ}(\mathbb{Z}_n,S)$ is integral if and only if $S =\bigcup\limits_{d\in \mathscr{D}}G_n(d)$, where $\mathscr{D} \subseteq \{ d :d\mid n\}$.
	\end{theorem}

Let $n \equiv 0 \Mod 3$. For a divisor $d$ of $\frac{n}{3}, r\in \{0,1,2\}$ and $g\in \mathbb{Z}$, define the following sets:
\begin{align*}
&  M_{n,3}^r(d)=\{dk: 0\leq dk \leq n-1 , k \equiv r \Mod 3 \},\\
&  G_{n,3}^r(d)=\{ dk: 1\leq dk \leq n-1 , \gcd(dk,n )= d,k\equiv r \Mod 3 \},\\
&  D_{g,3}= \{ k: k \text{ divides } g, k \not\equiv 0 \Mod 3\}~\text{ and}\\
&  D_{g,3}^r=\{ k: k \text{ divides } g, k \equiv r \Mod 3\} .
\end{align*}
It is clear $D_{g,3}=D_{g,3}^1 \bigcup\limits D_{g,3}^2$. 
 
\begin{lema}\label{FirstLemmaSetEqua} Let $n\equiv 0 \Mod 3$, $d$ divide $\frac{n}{3}$ and $g= \frac{n}{3d}$. Then the following hold:
		\begin{enumerate}[label=(\roman*)]
			\item $M_{n,3}^1(d) \cup M_{n,3}^2(d)=\bigcup\limits_{h\in D_{g,3}} G_n(hd) $;
			\item $M_{n,3}^0(d)= M_n(3d)\cup \{ 0\}$.
		\end{enumerate}
	\end{lema}
	\begin{proof} 
		\begin{enumerate}[label=(\roman*)]
			\item Let $dk\in M_{n,3}^1(d) \cup M_{n,3}^2(d)$. Lemma~\ref{BasicProp} gives $M_{n,3}^1(d) \cup M_{n,3}^2(d) \subseteq M_n(d)=\bigcup\limits_{h\mid 3g} G_n(hd)$. So there exists a divisor $h$ of $3g$ such that $dk=\alpha hd$, for some $\alpha \in \mathbb{Z}$ with $\gcd(\alpha, \frac{3g}{h})=1$. Now we have $h=\frac{k}{\alpha}$ and $k$ is not a multiple of $3$, which imply that $h|g$ and $h$ is not a multiple of $3$. Thus $h\in D_{g,3}$, and so $dk \in \bigcup\limits_{h\in D_{g,3}} G_n(hd)$. Conversely, let $x\in \bigcup\limits_{h\in D_{g,3}} G_n(hd)$. Then there exists $h\in D_{g,3}$ such that $x=\alpha hd$, where $\alpha \in \mathbb{Z}$ and $\gcd(\alpha, \frac{3g}{h})=1$. Note that $\alpha $ and $h$ are not multiples of $3$. Thus $\alpha h \equiv 1$ or $2 \Mod 3$, and so $x\in M_{n,3}^1(d) \cup M_{n,3}^2(d)$.
			\item By definition, we have
			\begin{equation*}
				\begin{split}
					M_{n,3}^0(d)&=\{ dk: 0\leq dk \leq n-1, k=3\alpha \textnormal{ for some } \alpha \in \mathbb{Z}\}\\
					&= \{ 3\alpha d : 1\leq 3 \alpha d \leq n-1 \textnormal{ for some } \alpha \in \mathbb{Z}\}\cup \{ 0\}\\
					&= M_n(3d) \cup \{0\}.
				\end{split} 
			\end{equation*}
		\end{enumerate}
	\end{proof}
	
We now prove that $G_n(d)$ is a disjoint union of $G_{n,3}^1(d)$ and $G_{n,3}^2(d)$.

	\begin{lema}\label{SecLemmaSetEqua} Let $n\equiv 0 \Mod 3$, $d$ divide $\frac{n}{3}$ and $g= \frac{n}{3d}$. Then the following hold:
		\begin{enumerate}[label=(\roman*)]
			\item $G_{n,3}^1(d) \cap G_{n,3}^2(d)=\emptyset$;
			\item $G_n(d)=G_{n,3}^1(d) \cup G_{n,3}^2(d)$;
			\item $M_{n,3}^1(d) = \bigg( \bigcup  \limits_{h\in D_{g,3}^1} G_{n,3}^1(hd) \bigg) \cup \bigg( \bigcup\limits_{h\in D_{g,3}^2} G_{n,3}^2(hd) \bigg)$;
			\item $M_{n,3}^2(d) =\bigg( \bigcup\limits_{h\in D_{g,3}^1} G_{n,3}^2(hd) \bigg) \cup \bigg( \bigcup\limits_{h\in D_{g,3}^2} G_{n,3}^1(hd) \bigg)$.
		\end{enumerate}
	\end{lema}
	\begin{proof}
		\begin{enumerate}[label=(\roman*)]
			\item It is clear from the definitions of $G_{n,3}^1(d)$ and $G_{n,3}^2(d)$ that $G_{n,3}^1(d) \cap G_{n,3}^2(d)=\emptyset$.
			
			\item Since $G_{n,3}^1(d) \subseteq G_n(d)$ and $G_{n,3}^2(d) \subseteq G_n(d)$, we have $G_{n,3}^1(d) \cup G_{n,3}^2(d) \subseteq G_n(d)$. Conversely, let $x\in G_n(d)$. Then $x=d\alpha$ for some $\alpha$ satisfying $\gcd(\alpha, 3g)=1$, and so $\alpha \equiv 1$ or $2 \Mod 3$. Hence $G_n(d) \subseteq G_{n,3}^1(d) \cup G_{n,3}^2(d)$.
			
			\item Let $dk\in M_{n,3}^1(d)$ so that $k\equiv 1\Mod 3$. By Lemma ~\ref{FirstLemmaSetEqua}, there exists $h\in D_{g,3}$ satisfying $dk\in G_n(hd)$. Thus $dk \in G_{n,3}^1(hd)$ or $dk \in G_{n,3}^2(hd)$.\\ 
			\textbf{Case 1.} Assume that $h\equiv 1\Mod 3$. Let, if possible, $dk \in G_{n,3}^2(hd)$, that is, $dk=\alpha hd$ for some $\alpha \equiv 2 \Mod 3$ satisfying $\gcd(\alpha, \frac{3g}{h})=1$. Then we have $k=\alpha h \equiv 2 \Mod 3$, a contradiction. Hence $dk \in G_{n,3}^1(hd)$.
			
			\textbf{Case 2.} Assume that $h\equiv 2\Mod 3$. Proceeding as in Case 1, we get $k\equiv 2 \Mod 3$, a contradiction. Hence $dk \in G_{n,3}^2(hd)$. Thus $M_{n,3}^1(d) \subseteq \bigg( \bigcup\limits_{h\in D_{g,3}^1} G_{n,3}^1(hd) \bigg) \cup \bigg( \bigcup\limits_{h\in D_{g,3}^2} G_{n,3}^2(hd) \bigg)$. 
			
			Conversely, if $\alpha hd \in G_{n,3}^1(hd)$, where $h\in D_{g,3}^1$ and $\alpha \equiv 1 \Mod 3$, we get $\alpha h \equiv 1 \Mod 3$, that is, $\alpha h d \in M_{n,3}^1(d)$. Similarly, $\beta h d \in G_{n,3}^2(hd)$, where $h\in D_{g,3}^2$ and $\beta \equiv 1 \Mod 3$, imply that $ \beta h d \in M_{n,3}^1(d)$. Therefore $\bigg( \bigcup\limits_{h\in D_{g,3}^1} G_{n,3}^1(hd) \bigg) \cup \bigg( \bigcup\limits_{h\in D_{g,3}^2} G_{n,3}^2(hd) \bigg) \subseteq M_{n,3}^1(d)$.
			
			\item The proof of this part is similar to the proof of Part (iii).
		\end{enumerate}
	\end{proof}

The \textit{cyclotomic polynomial} $\Phi_n(x)$ is the monic polynomial whose zeros are the primitive $n^{th}$ roots of unity. That is,
	$$\Phi_n(x)= \prod_{a\in G_n(1)}(x-\omega_n^a).$$ 
	Clearly, the degree of $\Phi_n(x)$ is $\varphi(n)$, where $\varphi$ denotes the Euler $\varphi$-function. It is well known that the cyclotomic polynomial $\Phi_n(x)$ is monic and irreducible in $\mathbb{Z}[x]$.  See \cite{numbertheory} for more details on cyclotomic polynomials.

	The polynomial $\Phi_n(x)$ is irreducible over $\mathbb{Q}(\omega_3)$ if and only if $[\mathbb{Q}(\omega_3,\omega_n) : \mathbb{Q}(\omega_3)]= \varphi(n)$. Also, $ \mathbb{Q}(\omega_n)$ does not contain the number $\omega_3$ if and only if $n\not\equiv 0 \Mod 3$. Thus, if $n\not\equiv 0 \Mod 3$ then $[\mathbb{Q}(\omega_3,\omega_n):\mathbb{Q}(\omega_n) ]=2=[\mathbb{Q}(\omega_3), \mathbb{Q}]$, and therefore 
	$$[\mathbb{Q}(\omega_3,\omega_n) : \mathbb{Q}(\omega_3)]=\frac{[\mathbb{Q}(\omega_3,\omega_n) : \mathbb{Q}(\omega_n)] \times [\mathbb{Q}(\omega_n) : \mathbb{Q}]}{ [\mathbb{Q}(\omega_3) : \mathbb{Q}]}= [\mathbb{Q}(\omega_n) : \mathbb{Q}]= \varphi(n).$$
Further, if $n\equiv 0 \Mod 3$ then	$ \mathbb{Q}(\omega_3,\omega_n)= \mathbb{Q}(\omega_n)$, and so 
$$[\mathbb{Q}(\omega_3,\omega_n) : \mathbb{Q}(\omega_3)] = \frac{[\mathbb{Q}(\omega_3,\omega_n) : \mathbb{Q}]}{[\mathbb{Q}(\omega_3) : \mathbb{Q}]}=\frac{\varphi(n)}{2}.$$
Note that $\mathbb{Q}(\omega_3)=\mathbb{Q}(\omega_6)=\mathbb{Q}(i\sqrt{3})$. Therefore $\Phi_n(x)$ is irreducible over $\mathbb{Q}(\omega_3), \mathbb{Q}(\omega_6)$ or $\mathbb{Q}(i\sqrt{3})$ if and only if $n\not\equiv 0 \Mod 3$.
	
Let $n\equiv 0 \Mod 3$. From Lemma~\ref{SecLemmaSetEqua}, we know that $G_{n}(1)$ is a disjoint union of $G_{n,3}^1(1)$ and $G_{n,3}^2(1)$. Define 
$$\Phi_{n,3}^{1}(x)= \prod_{a\in G_{n,3}^1(1)}(x-\omega_n^a)~~ \textnormal{ and } ~~\Phi_{n,3}^2(x)= \prod_{a\in G_{n,3}^2(1)}(x-\omega_n^a).$$ 
It is clear from the definition that $\Phi_n(x)=\Phi_{n,3}^1(x)\Phi_{n,3}^2(x)$.

	\begin{lema}\label{PolyDiviCyclo} If $n\equiv 0\Mod 3$ then the following hold:
		\begin{enumerate}[label=(\roman*)]
			\item $x^{\frac{n}{3}} -\omega_3= \prod\limits_{h\in D_{\frac{n}{3},3}^1}\Phi_{\frac{n}{h},3}^1(x) \prod\limits_{h\in D_{\frac{n}{3},3}^2} \Phi_{\frac{n}{h},3}^2(x)$; and
			\item $x^{\frac{n}{3}} - \omega_3^2= \prod\limits_{h\in D_{\frac{n}{3},3}^1}\Phi_{\frac{n}{h},3}^2(x) \prod\limits_{h\in D_{\frac{n}{3},3}^2} \Phi_{\frac{n}{h},3}^1(x)$.
		\end{enumerate}
	\end{lema}
	\begin{proof} 
		\begin{enumerate}[label=(\roman*)]
			\item We have $|M_{n,3}^1(1)|=\frac{n}{3}$ and $\omega_n^a$ is a root of $x^{\frac{n}{3}} - \omega_3 $ for each $a\in M_{n,3}^1(1)$. Therefore						
			\begin{equation*}
				\begin{split}
					x^{\frac{n}{3}} -\omega_3 &= \prod_{a\in M_{n,3}^1(1)} (x-\omega_n^a)\\
					&= \prod_{h\in D_{\frac{n}{3},3}^1} \prod_{a\in G_{n,3}^1(h)}(x-\omega_n^a) \prod_{h\in D_{\frac{n}{3},3}^2} \prod_{a\in G_{n,3}^2(h)}(x-\omega_n^a), \hspace{1cm} \textnormal{using Lemma } ~\ref{SecLemmaSetEqua}\\
					&= \prod_{h\in D_{\frac{n}{3},3}^1} \prod_{a\in hG_{\frac{n}{h},3}^1(1)}(x-\omega_n^a) \prod_{h\in D_{\frac{n}{3},3}^2} \prod_{a\in hG_{\frac{n}{h},3}^2(1)}(x-\omega_n^a)\\
					&= \prod_{h\in D_{\frac{n}{3},3}^1} \prod_{a\in G_{\frac{n}{h},3}^1(1)}(x-(\omega_n^h)^a) \prod_{h\in D_{\frac{n}{3},3}^2} \prod_{a\in G_{\frac{n}{h},3}^2(1)}(x-(\omega_n^h)^a)\\
					&= \prod_{h\in D_{\frac{n}{3},3}^1}\Phi_{\frac{n}{h},3}^1(x) \prod_{h\in D_{\frac{n}{3},3}^2} \Phi_{\frac{n}{h},3}^2(x).
				\end{split} 
			\end{equation*} 
			In the last equality, we have used the fact that $\omega_n^h=\exp(\frac{2\pi i}{n/h})$ is a primitive $\frac{n}{h}$-th root of unity.
\item We have $|M_{n,3}^2(1)|=\frac{n}{3}$ and $\omega_n^a$ is a root of $x^{\frac{n}{3}} -\omega_3^2 $ for each $a\in M_{n,3}^2(1)$. Now the proof is similar to the proof of Part $(i)$.
\end{enumerate}
\end{proof}

	\begin{corollary}\label{NewCycloPolyMonic}
		Let $n\equiv 0\Mod 3$. Then $\Phi_{n,3}^1(x)$ and $\Phi_{n,3}^2(x)$ are monic polynomials in $\mathbb{Z}(\omega_3)[x]$ of degree $\varphi(n)/2$.
	\end{corollary} 
	\begin{proof}
		By definition, $\Phi_{n,3}^1(x)$ and $\Phi_{n,3}^2(x)$ are monic polynomials. Also, $G_n(1)=G_{n,3}^1(1)\cup G_{n,3}^2(1)$ is a disjoint union and that $|G_{n,3}^1(1)|=|G_{n,3}^2(1)|$. Therefore the polynomials $\Phi_{n,3}^1(x)$ and $\Phi_{n,3}^2(x)$ are of degree $\frac{\varphi(n)}{2}$. Now apply induction on $n$ to show that $\Phi_{n,3}^1(x), \Phi_{n,3}^2(x)\in \mathbb{Z}(\omega_3)[x]$. For $n=3$, the polynomials $\Phi_{3,3}^1(x)= x-\omega_3$ and $\Phi_{3,3}^2(x)=x-\omega_3^2$ are clearly in $\mathbb{Z}(\omega_3)[x]$. Assume  that $\Phi_{k,3}^1(x)$ and $\Phi_{k,3}^2(x)$ are in $\mathbb{Z}(\omega_3)[x]$ for each $k<n$ and $k\equiv 0\Mod 3$. By Lemma~\ref{PolyDiviCyclo}, $\Phi_{n,3}^1(x) = \frac{x^{\frac{n}{3}} -\omega_3}{f(x)}$ and $\Phi_{n,3}^2(x) = \frac{x^{\frac{n}{3}} -\omega_3^2}{g(x)}$. By induction hypothesis, $f(x)$ and $g(x)$ are monic polynomials in $\mathbb{Z}(\omega_3)[x]$. It follows by ``long division'' that $\Phi_{n,3}^1(x)\in \mathbb{Z}(\omega_3)[x]$ and $\Phi_{n,3}^2(x)\in \mathbb{Z}(\omega_3)[x]$.
	\end{proof}

	\begin{theorem}
		Let $n\equiv 0\Mod 3$. Then $\Phi_{n,3}^1(x)$ and $\Phi_{n,3}^2(x)$ are irreducible monic polynomials in $\mathbb{Q}(\omega_3)[x]$ of degree $\frac{\varphi(n)}{2}$.
	\end{theorem}
	\begin{proof}
		In view of Corollary~\ref{NewCycloPolyMonic}, we only need to show the irreducibility of $\Phi_{n,3}^1(x)$ and $\Phi_{n,3}^2(x)$ in $\mathbb{Q}(\omega_3)$. For $n\equiv 0\Mod 3$, we have $[\mathbb{Q}(\omega_3,\omega_n) : \mathbb{Q}(\omega_3)] = \frac{\varphi(n)}{2}$. So there is a unique irreducible monic polynomial $p(x)\in \mathbb{Q}(\omega_3)[x]$ of degree $\frac{\varphi(n)}{2}$ having $\omega_n$ as a root. Since $\omega_n$ is also a root of $\Phi_{n,3}^1(x)$, we have that $\Phi_{n,3}^1(x)=p(x)f(x)$ for some $f(x)\in \mathbb{Q}(\omega_3)[x]$. Also $\Phi_{n,3}^1(x)$ is a monic polynomial of degree $\varphi(n)/2$, and so $f(x)=1$. Hence $\Phi_{n,3}^1(x)=p(x)$ is irreducible. Similarly, $[\mathbb{Q}(\omega_3,\omega_n^a) : \mathbb{Q}(\omega_3)] = \frac{\varphi(n)}{2}$ for $a\in G_{n,3}^2(1)$. This, along with Corollary~\ref{NewCycloPolyMonic}, give that $\Phi_{n,3}^2(x)$ is irreducible.
	\end{proof}

\section{A sufficient condition for the HS-integrality of oriented circulant graphs}\label{sufficient}
	
 In this section, we obtain a sufficient condition for the HS-integrality of oriented circulant graphs on $n$ vertices. Throughout this section, we consider $S$ to be a skew-symmetric subset of $\mathbb{Z}_n$.
	
	\begin{lema}\label{Sqrt3ZeroSum} Let $S$ be a skew-symmetric subset of $\mathbb{Z}_n$. If $\sum\limits_{k\in S} i\sqrt{3}(\omega_n^{jk}-\omega_n^{-jk}) =0$ for each \linebreak[4] $j\in \{0,1,...,n-1\}$ then $S=\emptyset$.
	\end{lema}
	\begin{proof} Let $A_S=(a_{uv})_{n\times n}$ be the matrix whose rows and columns are indexed by elements of $\mathbb{Z}_n$, where 
	$$a_{uv} = \left\{ \begin{array}{rl}
		i\sqrt{3} & \mbox{ if } v-u \in S \\
		-i\sqrt{3} & \mbox{ if } v-u \in S^{-1}\\
		0 &\textnormal{ otherwise.}
	\end{array}\right.$$ 
Observe that $A_S$ is a circulant matrix. Therefore the spectrum of $A_S$ is $\{ \alpha_j : j=0,1,...,n-1 \}$, where $\alpha_j= \sum\limits_{k\in S} i\sqrt{3}(\omega_n^{jk}-\omega_n^{-jk})$. Given that $\alpha_j=0$ for all $j=0,1,...,n-1$. This implies that all the entries of $A_S$ are zero, and hence $S=\emptyset$.
	\end{proof}

\begin{theorem}\label{CharaOfOrieCayGraphNot3}
		Let $n\not\equiv 0 \Mod 3$. Then the oriented circulant graph $\text{Circ}(\mathbb{Z}_n,S)$ is HS-integral if and only if $S=\emptyset$.
	\end{theorem}
	\begin{proof}
Let $G=\text{Circ}(\mathbb{Z}_n,S)$ and $Sp_H(G)=\{\mu_0,\mu_1,...,\mu_{n-1} \}$. Assume that $G$ is HS-integral. By Lemma~\ref{SpecMixCayGraph}, 
$$\mu_j=  \sum\limits_{k\in {S}} (\omega_6\omega_n^{jk} +  \omega_6^5 \omega_n^{-jk})\in \mathbb{Z} \textnormal{ for } 0\leq j\leq n-1.$$ 
Observe that $\omega_n$ is a root of the polynomial  $p(x)= \sum\limits_{k\in {S}}( \omega_6x^{jk} + \omega_6^5x^{j(n-k)})-\mu_j \in \mathbb{Q}(\omega_6)[x]$. Since $n\not\equiv 0 \Mod 3$, the polynomial $\Phi_n(x)$ is irreducible in $\mathbb{Q}(\omega_6)[x]$. Therefore $p(x)$ is a multiple of the irreducible polynomial $\Phi_n(x)$, and so $\omega_n^{-1}=\omega_n^{n-1}$ is also a root of $p(x)$, that is, $\mu_j=\mu_{n-j}$. Thus
			\begin{align*}
			0=\mu_j - \mu_{n-j}= \sum_{k \in S} (\omega_6 - \omega_6^5) \omega_n^{kj} + (\omega_6^5 - \omega_6)\omega_n^{-kj}
			=  \sum_{k \in S} i\sqrt{3}(\omega_n^{kj} -\omega_n^{-kj}).
			\end{align*} 
Therefore by Lemma \ref{Sqrt3ZeroSum}, $S=\emptyset$. Conversely, if $S=\emptyset$ then $\text{Circ}(\mathbb{Z}_n,S)$ has no edges, and hence all its eigenvalues are zero. 
	\end{proof}

Theorem~\ref{CharaOfOrieCayGraphNot3} characterizes integral oriented circulant graphs for the case $n\not\equiv 0 \Mod 3$.

\begin{lema}\label{CoroNewResSecKindHerInt}
Let $S$ be a skew-symmetric subset of $\mathbb{Z}_n$ and $k\in \mathbb{N}$. Then 
$$\sum\limits_{q\in S} \omega_k\omega_n^{jq} + \sum\limits_{q\in S^{-1}} \omega_k^{k-1}\omega_n^{jq} \in \mathbb{Z} \textnormal{ for each } j\in \{0,1,...,n-1\}$$ 
if and only if 
$$\sum\limits_{q\in S} \omega_{k}^{k-1}\omega_n^{jq} + \sum\limits_{q\in S^{-1}} \omega_k\omega_n^{jq} \in \mathbb{Z} \textnormal{ for each } j\in \{0,1,...,n-1\}.$$
\end{lema}
\begin{proof}
Let $H_S=[h_{uv}]_{n\times n}$ be the matrix, whose rows and columns are indexed by elements of $\mathbb{Z}_n$, where 
	$$h_{uv} = \left\{ \begin{array}{rl}
		\omega_k & \mbox{ if } v-u \in S\\
		\omega_k^{k-1} & \mbox{ if } v-u \in S^{-1}\\
		0 &\textnormal{ otherwise.}
	\end{array}\right.$$ 
 Since $H_S$ is a circulant matrix, $\sum\limits_{q\in S} \omega_{k}\omega_n^{jq} + \sum\limits_{q\in S^{-1}} \omega_k^{k-1}\omega_n^{jq}$ is an eigenvalue of $H_S$ for all $j\in \{0,1,...,n-1 \}$. Let $\overline{H}_S$ be obtained by taking the complex conjugate of the corresponding entry of $H_S$. Note that $\overline{H}_S= H_{S^{-1}}$. The result follows from the fact that the eigenvalues of $H_S$ are integers if and only if the eigenvalues of $\overline{H}_S$ are integers.
\end{proof}

\begin{lema}\label{NewTooSmallLemma} Let $n\equiv 0 \Mod 3$, $d$ divide $\frac{n}{3}$ and $g= \frac{n}{3d}$. Then $ \sum\limits_{q\in M_{n,3}^0(d)} \omega_n^{jq}\in \mathbb{Z} $ for $j\in \{0,1,...,n-1\}$. 
	\end{lema}
	\begin{proof} Using Lemma~\ref{FirstLemmaSetEqua} and Lemma~\ref{BasicProp}, we get $M_{n,3}^0(d)=\bigcup\limits_{h\mid g} G_n(3hd)\bigcup\limits \{0\}$. Therefore
		\begin{equation*}\label{eq1}
			\begin{split}
				\sum\limits_{q\in M_{n,3}^0(d)} \omega_n^{jq} =1+\sum\limits_{h\mid g} \sum\limits_{q\in G_n(3hd)} \omega_n^{jq}.
			\end{split} 
		\end{equation*}
By Theorem \ref{2006integral}, the eigenvalue $ \sum\limits_{q\in G_n(3hd)} \omega_n^{jq}$ of the circulant graph $\text{Circ}(\mathbb{Z}_n, G_n(3hd))$ is an integer for each $h \mid g$. Thus $\sum\limits_{q\in M_n^0(d)} \omega_n^{jq}\in \mathbb{Z}$ for each $j\in \{0,1,...,n-1\}$.
	\end{proof}
	
		\begin{lema}\label{NewSmallLemma}
	Let $n\equiv 0 \Mod 3$, $d$ divide $\frac{n}{3}$ and $g= \frac{n}{3d}$. Then $$\sum\limits_{q\in M_{n,3}^2(d)} \omega_3 \omega_n^{jq} + \sum\limits_{q\in M_{n,3}^1(d)}\omega_3^2 \omega_n^{jq} \in \mathbb{Z} \mbox{ for each } j\in \{0,1,...,n-1\}.$$
	\end{lema}
	\begin{proof} It is enough to show that $ \sum\limits_{q\in M_{n,3}^2(d)} \omega_3 z^{q} + \sum\limits_{q\in M_{n,3}^1(d)}\omega_3^2 z^{q}$ is an integer for all $z \in \mathbb{C}$ satisfying $z^n=1$. We have
		\begin{align}
				z^n-1&=(z^d-\omega_3)\bigg(\sum\limits_{k=1}^{3g} \omega_3^{k-1} z^{n-dk} \bigg)\nonumber\\
				&=(z^d-\omega_3)\bigg(\sum\limits_{k=1}^{3g} \omega_3^{k} z^{n-dk} \bigg) \omega_3^2\nonumber\\
				&=(z^d-\omega_3) \bigg( \sum\limits_{\substack{1\leq k \leq 3g \\ k\equiv 0 \Mod 3}} z^{n-dk} + \sum\limits_{\substack{1\leq k \leq 3g \\ k\equiv 1 \Mod 3}} \omega_3 z^{n-dk} + \sum\limits_{\substack{1\leq k \leq 3g \\ k\equiv 2 \Mod 3}} \omega_3^2 z^{n-dk}\bigg) \omega_3^2\nonumber\\
				&=(z^d-\omega_3) \bigg( \sum\limits_{q\in M_{n,3}^0(d)} z^{q} + \sum\limits_{q\in M_{n,3}^2(d)} \omega_3 z^{q} + \sum\limits_{q\in M_{n,3}^1(d)} \omega_3^2 z^{q}\bigg) \omega_3^2.\label{nowEqAdd}
		\end{align} 
Let $z\in \mathbb{C}$ such that $z^n=1$. By Equation (\ref{nowEqAdd}), we get two possible cases.\\
		\textbf{Case 1.} Assume that $z^d-\omega_3=0$.\\
		If $q\in M_{n,3}^1(d)$ then $q=(3y_1+1)d$ for some $y_1\in \mathbb{Z}$. So $z^q= z^{(3y_1+1)d}=\omega_3^{3y_1+1}=\omega_3$.\\
		If $q\in M_{n,3}^2(d)$ then $q=(3y_2+2)d$ for some $y_2\in \mathbb{Z}$. So $z^q= z^{(3y_2+2)d}=\omega_3^{3y_2+2}=\omega_3^2$. Thus
		$$\sum\limits_{q\in M_{n,3}^2(d)} \omega_3 z^q + \sum\limits_{q\in M_{n,3}^1(d)}\omega_3^2 z^q =\sum\limits_{q\in M_{n,3}^2(d)}1 + \sum\limits_{q\in M_{n,3}^1(d)}1 = |M_{n,3}^1(d) \cup M_{n,3}^2(d)| \in \mathbb{Z}.$$
\noindent \textbf{Case 2.} Assume that $z^d-\omega_3\neq0$. Then 
\begin{align*}
&\sum\limits_{q\in M_{n,3}^0(d)} z^{q} + \sum\limits_{q\in M_{n,3}^2(d)} \omega_3 z^{q} + \sum\limits_{q\in M_{n,3}^1(d)} \omega_3^2 z^{q}=0\\
\Rightarrow &\sum\limits_{q\in M_{n,3}^2(d)} \omega_3 z^{q} + \sum\limits_{q\in M_{n,3}^1(d)} \omega_3^2 z^{q} = \sum\limits_{q\in M_{n,3}^0(d)} z^{q} \in \mathbb{Z}, \textnormal{ by Lemma } \ref{NewTooSmallLemma}. 
\end{align*}
	\end{proof}

For $n \equiv 0 \Mod 3$ and $j \in \{0,1,\ldots,n-1\}$, define 
$$Z_n^1(j)= \sum\limits_{q\in G_{n,3}^1(1)} \left(\omega_3\omega_n^{jq} + \omega_3^2 \omega_n^{-jq}\right)~\text{ and }~Z_n^2(j)= \sum\limits_{q\in G_{n,3}^2(1)}\left(\omega_3\omega_n^{jq} + \omega_3^2 \omega_n^{-jq}\right).$$ 
By Lemma \ref{CoroNewResSecKindHerInt}, $Z_n^1(j)$ is an integer if and only if $Z_n^2(j)$ is an integer.

Note that $G_n(d)=dG_{\frac{n}{d}}(1)$, $G_{n,3}^1(d)=dG_{\frac{n}{d},3}^1(1)$ and $G_{n,3}^2(d)=dG_{\frac{n}{d},3}^2(1)$. Therefore, if $d$ is a divisor of $\frac{n}{3}$ then
		\begin{equation*}\label{}
			\begin{split}
	Z_{\frac{n}{d}}^1(j) =\sum_{q \in G_{\frac{n}{d},3}^1(1)} \left[\omega_3(\omega_{\frac{n}{d}})^{jq}+ \omega_3^2(\omega_{\frac{n}{d}})^{-jq}\right] &=  \sum_{q\in G_{\frac{n}{d},3}^1(1)} \left[\omega_3(\omega_n^d)^{jq}+ \omega_3^2(\omega_n^d)^{-jq}\right]\\
					&=  \sum_{q\in dG_{\frac{n}{d},3}^1(1)}\left( \omega_3\omega_n^{jq}+ \omega_3^2\omega_n^{-jq}\right)\\
					&=  \sum_{q\in G_{n,3}^1(d)} \left(\omega_3\omega_n^{jq}+ \omega_3^2\omega_n^{-jq}\right).
			\end{split} 
		\end{equation*}			
Similarly, if $d$ be a divisor of $\frac{n}{3}$ then	
$$Z_{\frac{n}{d}}^2(j) = \sum_{q\in G_{n,3}^2(d)} \left(\omega_3\omega_n^{jq}+ \omega_3^2\omega_n^{-jq}\right).$$

	\begin{lema}\label{SuffiCondMainTheo} 
		Let $n\equiv 0 \Mod 3$ and $d$ divide $\frac{n}{3}$. Then $Z_{\frac{n}{d}}^1(j)$ is an integer for each $j \in \{0,1,\ldots,n-1\}$.
	\end{lema}
	\begin{proof}
	Let $d_1=\frac{n}{3},d_2,...,d_r=1$ be the positive divisors of $\frac{n}{3}$ in decreasing order. Apply induction on $k$ to prove that $Z_{\frac{n}{d_k}}(j)$ is an integer for each $j\in \mathbb{Z}$. For $k=1$, $Z_{\frac{n}{d_1}}^1(j)= Z_3^1(j)=\sum\limits_{q\in G_{3,3}^1(1)}\left(\omega_3\omega_3^{jq} + \omega_3^2 \omega_3^{-jq}\right)$ is an integer for each $j\in \{0,1,\ldots,n-1\}$. Assume that $Z_{\frac{n}{d_k}}^1(j)$ is an integer for each $j\in \mathbb{Z}$, where $1\leq k <r$. Let $n=3d_{k+1}g_{k+1}$ for some $d_{k+1},g_{k+1}\in \mathbb{Z}$. By Lemma ~\ref{SecLemmaSetEqua}, we have
		\begin{equation}\label{eq2}
			\begin{split}
				M_{n,3}^1(d_{k+1}) =\bigg(\bigcup\limits_{h\in D_{g_{k+1},3}^1} G_{n,3}^1(hd_{k+1}) \bigg) \cup \bigg( \bigcup\limits_{h\in D_{g_{k+1},3}^2} G_{n,3}^2(hd_{k+1}) \bigg).
			\end{split} 
		\end{equation}
		Note that $hd_{k+1}$ is also a divisor of $\frac{n}{3}$. If $h>1$ then $hd_{k+1}=d_s$ for some $s<k+1$, and so by induction hypothesis $Z_{\frac{n}{hd_{k+1}}}^1(j)$ and $Z_{\frac{n}{hd_{k+1}}}^2(j)$ are integers for each $j \in \mathbb{Z}$. Note that the unions in (\ref{eq2}) is disjoint. We have
		\begin{equation*}
			\begin{split}
				\sum\limits_{q\in M_{n,3}^1(d_{k+1})}\left(\omega_3\omega_n^{jq} + \omega_3^2 \omega_n^{-jq}\right) =& \sum\limits_{q\in G_{n,3}^1(d_{k+1})}\left(\omega_3\omega_n^{jq} +  \omega_3^2 \omega_n^{-jq}\right)  \\
				&+ \sum\limits_{h\in D_{g_{k+1},3}^1, h> 1} \bigg( \sum\limits_{q\in G_{n,3}^1(hd_{k+1})}\left(\omega_3\omega_n^{jq} +  \omega_3^2\omega_n^{-jq}\right) \bigg)\\
				&+ \sum\limits_{h\in D_{g_{k+1},3}^2} \bigg( \sum\limits_{q\in G_{n,3}^2(hd_{k+1})}\left(\omega_3 \omega_n^{jq} +  \omega_3^2 \omega_n^{-jq} \right)\bigg).
			\end{split} 
		\end{equation*}
		Therefore
		\begin{align}
		\sum\limits_{q\in M_{n,3}^1(d_{k+1})} \left(\omega_3\omega_n^{jq} + \omega_3^2 \omega_n^{-jq}\right) = Z_{\frac{n}{d_{k+1}}}^1(j) +  \sum\limits_{h\in D_{g_{k+1},3}^1, h> 1}	Z_{\frac{n}{hd_{k+1}}}^1(j)	+ \sum\limits_{h\in D_{g_{k+1},3}^2} Z_{\frac{n}{hd_{k+1}}}^2(j),\nonumber 
		\end{align} 
		and so 
		\begin{align}\label{eq4}
		 Z_{\frac{n}{d_{k+1}}}^1(j)  = \sum\limits_{q\in M_{n,3}^1(d_{k+1})} \left(\omega_3\omega_n^{jq} + \omega_3^2 \omega_n^{-jq}\right) -  \sum\limits_{h\in D_{g_{k+1},3}^1, h> 1}	Z_{\frac{n}{hd_{k+1}}}^1(j)	- \sum\limits_{h\in D_{g_{k+1},3}^2} Z_{\frac{n}{hd_{k+1}}}^2(j).
		\end{align}
		By Lemma~\ref{CoroNewResSecKindHerInt} and Lemma~\ref{NewSmallLemma}, the first summand in the right of (\ref{eq4}) is an integer, and by induction hypothesis the other two summands are also integers. Hence $Z_{\frac{n}{d_{k+1}}}^1(j)$ is an integer for each $j\in \mathbb{Z}$. Thus the proof is complete by induction.
	\end{proof}

\begin{corollary}\label{SuffiCondMainTheoCoro}
		Let $n\equiv 0 \Mod 3$ and $d$  a divisor of  $\frac{n}{3}$. If $S \in \{G_{n,3}^1(d),G_{n,3}^2(d) \}$, then the sum $\sum\limits_{q\in S}(\omega_3\omega_n^{qj}+ \omega_3^2\omega_n^{-qj})$ is an integer for each $j\in \{0,1,...,n-1\}$. 
	\end{corollary}
	\begin{proof}
	The proof follows from Lemma \ref{CoroNewResSecKindHerInt} and Lemma \ref{SuffiCondMainTheo}. 
	\end{proof}

\begin{corollary}\label{SuffiCondMainTheoCoroNew}
	Let $n\equiv 0 \Mod 3$ and $d$ a divisor of $\frac{n}{3}$. If $S \in \{G_{n,3}^1(d),G_{n,3}^2(d) \}$, then the sum $\sum\limits_{q\in S}(\omega_6\omega_n^{qj}+ \omega_6^5\omega_n^{-qj})$ is an integer for each $j\in \{0,1,...,n-1\}$. 
\end{corollary}
\begin{proof}
Note that $\omega_3=-\omega_6^5,\omega_3^2=-\omega_6$ and $\left(G_{n,3}^1(d)\right)^{-1}=G_{n,3}^2(d)$. Therefore the proof follows from Corollary~\ref{SuffiCondMainTheoCoro}.
\end{proof}
	
	In the next result, we get a sufficient condition on the connection set $S$ for which the oriented circulant graph $\text{Circ}(\mathbb{Z}_n,S)$ is HS-integral. 
	
	\begin{theorem}\label{CoroSuffiCondiOriCay} 
		Let $S$ be a subset of $\mathbb{Z}_n$ such that
		\begin{equation*}
			S= \left\{ \begin{array}{ll}
				\emptyset & \text{ if } n\not\equiv 0\Mod 3 \\ 
				\bigcup\limits_{d\in \mathscr{D}}S_n(d) & \text{ if } n\equiv 0\Mod 3,
			\end{array}\right.
		\end{equation*}
		where $\mathscr{D} \subseteq \{ d: d\mid \frac{n}{3}\}$ and $S_n(d)\in \{ G_{n,3}^{1}(d) , G_{n,3}^{2}(d)\}$. Then the oriented circulant graph $\text{Circ}(\mathbb{Z}_n,S)$ is HS-integral. 
	\end{theorem}
	\begin{proof}
	The proof follows from Theorem~\ref{CharaOfOrieCayGraphNot3} and Corollary~\ref{SuffiCondMainTheoCoroNew}.
	\end{proof}
	
\section{Characterization of HS-integral mixed circulant graphs}\label{necessity-mixed}

In this section, we first characterize HS-integrality of oriented circulant graphs by proving the necessity of Theorem \ref{CoroSuffiCondiOriCay}. After that, we extend this characterization to mixed circulant graphs.  Recall that the case $n\not\equiv 0\Mod 3$ have already been considered in Theorem \ref{CharaOfOrieCayGraphNot3}. Now let $n\equiv 0\Mod 3$ and $d$ a divisor of $\frac{n}{3}$. Define $u(d)=\left[u(d)_1,\ldots,u(d)_n\right]^T$ to be the $n$-vector, where
	
	$$u(d)_k= \left\{ \begin{array}{rl}
		\omega_3 & \mbox{if } k\in G_{n,3}^1(d) \\
		\omega_3^2 & \mbox{if } k\in G_{n,3}^2(d)\\ 
		0 &   \mbox{otherwise}. 
	\end{array}\right.	$$ 
Let $E=[e_{st}]$ be the $n \times n$ matrix defined by $e_{st}=\omega_n^{st}$. Note that $E$ is an invertible matrix and $EE^*=nI_n$, where $E^*$ is the conjugate transpose of $E$. Using Theorem ~\ref{CoroSuffiCondiOriCay}, we get $E u(d) \in \mathbb{Z}^{n}$.

	\begin{lema}\label{MainLemmNeceCond}
		Let $n\equiv 0\Mod 3$ and $v\in \mathbb{Q}^n(\omega_3)$ such that $Ev \in \mathbb{Q}^n$. Let the coordinates of $v$ be indexed by the elements of $\mathbb{Z}_n$. Then
		\begin{enumerate}[label=(\roman*)]
			\item $\overline{v}_s=v_{n-s}$ for all $1\leq s \leq n$;
			\item if $d \mid \frac{n}{6}$ then $v_s=v_t$ for all $s,t\in G_{n,3}^1(d)$; and
			\item if $d \nmid \frac{n}{3}, d \mid n$ and $d<n$ then $v_r=v_{n-r}$ for all $r\in G_n(d)$.
		\end{enumerate} 
	\end{lema}
	\begin{proof}
	Let $u=Ev \in \mathbb{Q}^n$, where $v \in \mathbb{Q}^n(\omega_3)$. Then $v=\frac{1}{n}E^*u$, and so 
	\begin{align}
	v_s = \frac{1}{n} \sum_{j=1}^{n}u_j\omega_n^{s(n-j)}.\label{EqSixNecess}
	\end{align}
			\begin{enumerate}[label=(\roman*)]
			\item Taking complex conjugate of both sides in Equation~\ref{EqSixNecess}, we get $\overline{v}_s=v_{n-s}$ for all $1\leq s \leq n$.
			
			\item Assume that $d \mid \frac{n}{3}$ and $s,t\in G_{n,3}^1(d)$. Then $\frac{s}{d},\frac{t}{d}\in G_{\frac{n}{d},3}^1(1)$, and so $\omega_n^s$, $\omega_n^t$ are roots of $\Phi_{\frac{n}{d},3}^1(x)$. Note that $v_{s}=\frac{1}{n}\sum\limits_{j=1}^{n}u_j\omega_n^{s(n-j)} \in \mathbb{Q}(\omega_3)$. Therefore $\omega_n^s$ is a root of $p(x)=\frac{1}{n} \sum\limits_{j=1}^{n}u_jx^{n-j}-v_s \in \mathbb{Q}(\omega_3)[x]$. Hence $p(x)$ is a multiple of the irreducible monic polynomial $\Phi_{\frac{n}{d},3}^1(x)$, and so $\omega_n^t$ is also a root of $p(x)$, that is, $v_{s}=\frac{1}{n} \sum\limits_{j=1}^{n}u_j\omega_n^{t(n-j)}=v_t$.
			
			\item Assume that $d \nmid \frac{n}{3}, d \mid n$, $d<n$ and $r\in G_n(d)$. Then $r,n-r\in G_n(d)$, and so $\omega_n^r$, $\omega_n^{n-r}$ are roots of $\Phi_{\frac{n}{d}}(x)$. Since $v_{r}=\frac{1}{n} \sum\limits_{j=1}^{n}u_j\omega_n^{r(n-j)} \in \mathbb{Q}(\omega_3)$, we find that $\omega_n^r$ is a root of the polynomial $q(x)=\frac{1}{n} \sum\limits_{j=1}^{n}u_jx^{(n-j)}-v_r \in \mathbb{Q}(\omega_3)[x]$. Therefore $q(x)$ is a multiple of the irreducible monic polynomial $\Phi_{\frac{n}{d}}(x)$. Note that $\Phi_{\frac{n}{d}}(x)$ is irreducible over $\mathbb{Q}(\omega_3)$ as $\frac{n}{d} \not\equiv 0 \Mod 3$.
		Thus $\omega_n^{n-r}$ is also a root of $q(x)$, that is, $v_{r}=\frac{1}{n} \sum\limits_{j=1}^{n}u_j\omega_n^{(n-r)(n-j)}=v_{n-r}$. 
		\end{enumerate} 
		\end{proof}

	In the next theorem, we prove the converse of Theorem~\ref{CoroSuffiCondiOriCay}.

	\begin{theorem}\label{CharaOfOrieCayGraph}
Let $S$ be a skew-symmetric subset of $\mathbb{Z}_n$. Then the oriented circulant graph $\text{Circ}(\mathbb{Z}_n,S)$ is HS-integral if and only if 
	\begin{equation*}
			S= \left\{ \begin{array}{ll}
				\emptyset & \text{ if } n\not\equiv 0\Mod 3 \\ 
				\bigcup\limits_{d\in \mathscr{D}}S_n(d) & \text{ if } n\equiv 0\Mod 3,
			\end{array}\right.
		\end{equation*}
		where $\mathscr{D} \subseteq \{ d: d\mid \frac{n}{3}\}$ and $S_n(d)\in \{ G_{n,3}^{1}(d) , G_{n,3}^{2}(d)\}$. 
	\end{theorem}
	\begin{proof} We proved the sufficient part in Theorem~\ref{CoroSuffiCondiOriCay}. Assume that $\text{Circ}(\mathbb{Z}_n,S)$ is HS-integral. If $n\not\equiv 0\Mod 3$ then by Theorem \ref{CharaOfOrieCayGraphNot3}, $S= \emptyset$. Let $n\equiv 0\Mod 3$ and let $v$ be the vector of length $n$ defined by
			$$v_k= \left\{ \begin{array}{rl}
				-\omega_3^2 & \text{ if }  k\in S \\
				-\omega_3 & \text{ if }  n-k\in S\\ 
				0 &  \text{ otherwise}. 
			\end{array}\right.			$$
Note that each entry of $Ev$ is an HS-eigenvalue of $\text{Circ}(\mathbb{Z}_n,S)$, and so $Ev \in \mathbb{Z}^{n}$. Thus $v$ satisfies all the conditions of Lemma~\ref{MainLemmNeceCond}. 

By conditions $(i)$ and $(iii)$ of Lemma~\ref{MainLemmNeceCond}, if $d \nmid \frac{n}{3}, d \mid n$ and $d<n$ then $v_r=v_{n-r}\in \Rl$ for all $r\in G_n(d)$. Also, by definition of $v$ we have, $v_r\in \Rl$ if and only if $r\notin S\cup S^{-1}$. Therefore $S\subseteq \bigcup\limits_{d \mid \frac{n}{3}} \left[G_{n,3}^1(d) \cup G_{n,3}^2(d) \right]$. 

Further, if $r\in S\cap G_{n,3}^1(d)$ for some $d \mid \frac{n}{3}$, again by condition $(ii)$ of Lemma~\ref{MainLemmNeceCond}, we have $G_{n,3}^1(d)\subseteq S$. Similarly, if $r\in S\cap G_{n,3}^2(d)$ for some $d \mid \frac{n}{3}$ then $G_{n,3}^2(d)\subseteq S$. Thus there exists $\mathscr{D} \subseteq \{ d: d\mid \frac{n}{3}\}$ such that 
\begin{equation*}
			S= \left\{ \begin{array}{ll}
				\emptyset & \text{ if } n\not\equiv 0\Mod 3 \\ 
				\bigcup\limits_{d\in \mathscr{D}}S_n(d) & \text{ if } n\equiv 0\Mod 3,
			\end{array}\right.
		\end{equation*}
		where $S_n(d)\in \{ G_{n,3}^{1}(d) , G_{n,3}^{2}(d)\}$. 
	\end{proof}

The following example illustrates Theorem~\ref{CharaOfOrieCayGraph}.

\begin{ex} Consider $\Gamma= \mathbb{Z}_{12}$ and $S=\{ 2,5,11\}$. The oriented graph $\text{Circ}(\mathbb{Z}_{12}, S)$ is shown in Figure~\ref{a}. We see that $G_{12,3}^2(1)=\{5,11\}$ and $G_{12,3}^1(2)=\{2\}$. Therefore $S=G_{12,3}^2(1) \cup G_{12,3}^1(2)$, and hence $\text{Circ}(\mathbb{Z}_{12}, S)$ is HS-integral. Further, using Corollary~\ref{SpecMixCayGraph}, the HS-eigenvalues of  $\text{Circ}(\mathbb{Z}_{12}, S)$ are obtained as
$$\mu_j=\cos\left(\frac{\pi j}{3}\right) + \cos\left(\frac{5\pi j}{6}\right) + \cos\left(\frac{11\pi j}{6}\right) - \sqrt{3} \left[ \sin\left(\frac{\pi j}{3}\right) + \sin\left(\frac{5\pi j}{6}\right) + \sin\left(\frac{1\pi j}{6}\right) \right],$$
$\text{for each } j \in \mathbb{Z}_{12}$. We find that $\mu_0=3, \mu_1=-1, \mu_2=2, \mu_3=-1, \mu_4=3, \mu_5=2, \mu_6=-1,$ $ \mu_7=-1, \mu_8=-6, \mu_9=-1, \mu_{10}=-1$ and $\mu_{11}=2$. Thus all the HS-eigenvalues of $\text{Circ}(\mathbb{Z}_{12}, S)$ are integers. \qed
\end{ex}

\begin{figure}[ht]
\centering
\hfill
\begin{subfigure}{0.45\textwidth}
\tikzset{every picture/.style={line width=0.75pt}} 
\begin{tikzpicture}[x=0.23pt,y=0.23pt,yscale=-1,xscale=1]
\draw   (200,440.5) .. controls (200,285.58) and (325.58,160) .. (480.5,160) .. controls (635.42,160) and (761,285.58) .. (761,440.5) .. controls (761,595.42) and (635.42,721) .. (480.5,721) .. controls (325.58,721) and (200,595.42) .. (200,440.5) -- cycle ;
\draw  [fill={rgb, 255:red, 255; green, 255; blue, 255 }  ,fill opacity=1 ] (454.13,161) .. controls (454.13,146.43) and (465.93,134.63) .. (480.5,134.63) .. controls (495.07,134.63) and (506.88,146.43) .. (506.88,161) .. controls (506.88,175.57) and (495.07,187.38) .. (480.5,187.38) .. controls (465.93,187.38) and (454.13,175.57) .. (454.13,161) -- cycle ;
\draw  [fill={rgb, 255:red, 255; green, 255; blue, 255 }  ,fill opacity=1 ] (593.13,199) .. controls (593.13,184.43) and (604.93,172.63) .. (619.5,172.63) .. controls (634.07,172.63) and (645.88,184.43) .. (645.88,199) .. controls (645.88,213.57) and (634.07,225.38) .. (619.5,225.38) .. controls (604.93,225.38) and (593.13,213.57) .. (593.13,199) -- cycle ;
\draw  [fill={rgb, 255:red, 255; green, 255; blue, 255 }  ,fill opacity=1 ] (698.13,303) .. controls (698.13,288.43) and (709.93,276.63) .. (724.5,276.63) .. controls (739.07,276.63) and (750.88,288.43) .. (750.88,303) .. controls (750.88,317.57) and (739.07,329.38) .. (724.5,329.38) .. controls (709.93,329.38) and (698.13,317.57) .. (698.13,303) -- cycle ;
\draw  [fill={rgb, 255:red, 255; green, 255; blue, 255 }  ,fill opacity=1 ] (733.13,440) .. controls (733.13,425.43) and (744.93,413.63) .. (759.5,413.63) .. controls (774.07,413.63) and (785.88,425.43) .. (785.88,440) .. controls (785.88,454.57) and (774.07,466.38) .. (759.5,466.38) .. controls (744.93,466.38) and (733.13,454.57) .. (733.13,440) -- cycle ;
\draw  [fill={rgb, 255:red, 255; green, 255; blue, 255 }  ,fill opacity=1 ] (455.13,719) .. controls (455.13,704.43) and (466.93,692.63) .. (481.5,692.63) .. controls (496.07,692.63) and (507.88,704.43) .. (507.88,719) .. controls (507.88,733.57) and (496.07,745.38) .. (481.5,745.38) .. controls (466.93,745.38) and (455.13,733.57) .. (455.13,719) -- cycle ;
\draw  [fill={rgb, 255:red, 255; green, 255; blue, 255 }  ,fill opacity=1 ] (698.13,583) .. controls (698.13,568.43) and (709.93,556.63) .. (724.5,556.63) .. controls (739.07,556.63) and (750.88,568.43) .. (750.88,583) .. controls (750.88,597.57) and (739.07,609.38) .. (724.5,609.38) .. controls (709.93,609.38) and (698.13,597.57) .. (698.13,583) -- cycle ;
\draw  [fill={rgb, 255:red, 255; green, 255; blue, 255 }  ,fill opacity=1 ] (595.13,682) .. controls (595.13,667.43) and (606.93,655.63) .. (621.5,655.63) .. controls (636.07,655.63) and (647.88,667.43) .. (647.88,682) .. controls (647.88,696.57) and (636.07,708.38) .. (621.5,708.38) .. controls (606.93,708.38) and (595.13,696.57) .. (595.13,682) -- cycle ;
\draw  [fill={rgb, 255:red, 255; green, 255; blue, 255 }  ,fill opacity=1 ] (314.6,199.62) .. controls (314.51,185.05) and (326.25,173.17) .. (340.81,173.08) .. controls (355.38,172.98) and (367.26,184.72) .. (367.35,199.29) .. controls (367.44,213.85) and (355.71,225.73) .. (341.14,225.83) .. controls (326.58,225.92) and (314.69,214.18) .. (314.6,199.62) -- cycle ;
\draw  [fill={rgb, 255:red, 255; green, 255; blue, 255 }  ,fill opacity=1 ] (210.27,302.94) .. controls (210.18,288.38) and (221.91,276.49) .. (236.48,276.4) .. controls (251.04,276.31) and (262.92,288.04) .. (263.02,302.61) .. controls (263.11,317.18) and (251.37,329.06) .. (236.81,329.15) .. controls (222.24,329.24) and (210.36,317.51) .. (210.27,302.94) -- cycle ;
\draw  [fill={rgb, 255:red, 255; green, 255; blue, 255 }  ,fill opacity=1 ] (175.13,439.72) .. controls (175.04,425.15) and (186.77,413.27) .. (201.34,413.18) .. controls (215.9,413.09) and (227.78,424.82) .. (227.88,439.39) .. controls (227.97,453.95) and (216.23,465.84) .. (201.67,465.93) .. controls (187.1,466.02) and (175.22,454.29) .. (175.13,439.72) -- cycle ;
\draw  [fill={rgb, 255:red, 255; green, 255; blue, 255 }  ,fill opacity=1 ] (210.03,582.94) .. controls (209.93,568.37) and (221.67,556.49) .. (236.24,556.4) .. controls (250.8,556.3) and (262.68,568.04) .. (262.78,582.61) .. controls (262.87,597.17) and (251.13,609.05) .. (236.57,609.15) .. controls (222,609.24) and (210.12,597.5) .. (210.03,582.94) -- cycle ;
\draw  [fill={rgb, 255:red, 255; green, 255; blue, 255 }  ,fill opacity=1 ] (312.65,682.58) .. controls (312.56,668.02) and (324.29,656.13) .. (338.86,656.04) .. controls (353.43,655.95) and (365.31,667.68) .. (365.4,682.25) .. controls (365.49,696.82) and (353.76,708.7) .. (339.19,708.79) .. controls (324.62,708.88) and (312.74,697.15) .. (312.65,682.58) -- cycle ;
\draw  [draw opacity=0] (184.52,421.19) .. controls (184.52,421.19) and (184.52,421.19) .. (184.52,421.19) .. controls (125.3,386.99) and (106.92,307.96) .. (143.46,244.67) .. controls (180.01,181.37) and (257.64,157.77) .. (316.86,191.96) -- (250.69,306.57) -- cycle ; \draw   (184.52,421.19) .. controls (184.52,421.19) and (184.52,421.19) .. (184.52,421.19) .. controls (125.3,386.99) and (106.92,307.96) .. (143.46,244.67) .. controls (180.01,181.37) and (257.64,157.77) .. (316.86,191.96) ;
\draw  [draw opacity=0] (347.64,174.46) .. controls (347.64,174.46) and (347.64,174.46) .. (347.64,174.46) .. controls (347.64,106.08) and (406.89,50.65) .. (479.98,50.65) .. controls (553.07,50.65) and (612.33,106.08) .. (612.33,174.46) -- (479.98,174.46) -- cycle ; \draw   (347.64,174.46) .. controls (347.64,174.46) and (347.64,174.46) .. (347.64,174.46) .. controls (347.64,106.08) and (406.89,50.65) .. (479.98,50.65) .. controls (553.07,50.65) and (612.33,106.08) .. (612.33,174.46) ;
\draw  [draw opacity=0] (645.81,192.85) .. controls (645.81,192.85) and (645.81,192.85) .. (645.81,192.85) .. controls (705.03,158.66) and (782.67,182.26) .. (819.21,245.56) .. controls (855.76,308.85) and (837.38,387.89) .. (778.15,422.08) -- (711.98,307.46) -- cycle ; \draw   (645.81,192.85) .. controls (645.81,192.85) and (645.81,192.85) .. (645.81,192.85) .. controls (705.03,158.66) and (782.67,182.26) .. (819.21,245.56) .. controls (855.76,308.85) and (837.38,387.89) .. (778.15,422.08) ;
\draw  [draw opacity=0] (779.15,459.85) .. controls (779.15,459.85) and (779.15,459.85) .. (779.15,459.85) .. controls (838.38,494.04) and (856.76,573.07) .. (820.21,636.37) .. controls (783.67,699.67) and (706.03,723.27) .. (646.81,689.08) -- (712.98,574.46) -- cycle ; \draw   (779.15,459.85) .. controls (779.15,459.85) and (779.15,459.85) .. (779.15,459.85) .. controls (838.38,494.04) and (856.76,573.07) .. (820.21,636.37) .. controls (783.67,699.67) and (706.03,723.27) .. (646.81,689.08) ;
\draw  [draw opacity=0] (613.33,706.46) .. controls (613.33,706.46) and (613.33,706.46) .. (613.33,706.46) .. controls (613.33,774.85) and (554.07,830.28) .. (480.98,830.28) .. controls (407.89,830.28) and (348.64,774.85) .. (348.64,706.46) -- (480.98,706.46) -- cycle ; \draw   (613.33,706.46) .. controls (613.33,706.46) and (613.33,706.46) .. (613.33,706.46) .. controls (613.33,774.85) and (554.07,830.28) .. (480.98,830.28) .. controls (407.89,830.28) and (348.64,774.85) .. (348.64,706.46) ;
\draw  [draw opacity=0] (313.15,687.8) .. controls (253.93,721.99) and (176.3,698.39) .. (139.76,635.09) .. controls (103.21,571.79) and (121.59,492.76) .. (180.81,458.57) -- (246.98,573.18) -- cycle ; \draw   (313.15,687.8) .. controls (253.93,721.99) and (176.3,698.39) .. (139.76,635.09) .. controls (103.21,571.79) and (121.59,492.76) .. (180.81,458.57) ;
\draw  [draw opacity=0] (231.08,277.74) .. controls (196.89,218.52) and (220.48,140.89) .. (283.78,104.35) .. controls (347.08,67.8) and (426.11,86.18) .. (460.3,145.4) -- (345.69,211.57) -- cycle ; \draw   (231.08,277.74) .. controls (196.89,218.52) and (220.48,140.89) .. (283.78,104.35) .. controls (347.08,67.8) and (426.11,86.18) .. (460.3,145.4) ;
\draw  [draw opacity=0] (502.08,146.4) .. controls (502.08,146.4) and (502.08,146.4) .. (502.08,146.4) .. controls (536.27,87.18) and (615.3,68.8) .. (678.6,105.35) .. controls (741.9,141.89) and (765.49,219.52) .. (731.3,278.74) -- (616.69,212.57) -- cycle ; \draw   (502.08,146.4) .. controls (502.08,146.4) and (502.08,146.4) .. (502.08,146.4) .. controls (536.27,87.18) and (615.3,68.8) .. (678.6,105.35) .. controls (741.9,141.89) and (765.49,219.52) .. (731.3,278.74) ;
\draw  [draw opacity=0] (749.69,310.23) .. controls (749.69,310.23) and (749.69,310.23) .. (749.69,310.23) .. controls (818.07,310.23) and (873.5,369.48) .. (873.5,442.57) .. controls (873.5,515.66) and (818.07,574.92) .. (749.69,574.92) -- (749.69,442.57) -- cycle ; \draw   (749.69,310.23) .. controls (749.69,310.23) and (749.69,310.23) .. (749.69,310.23) .. controls (818.07,310.23) and (873.5,369.48) .. (873.5,442.57) .. controls (873.5,515.66) and (818.07,574.92) .. (749.69,574.92) ;
\draw  [draw opacity=0] (724.5,609.38) .. controls (724.5,609.38) and (724.5,609.38) .. (724.5,609.38) .. controls (758.69,668.6) and (735.09,746.23) .. (671.8,782.77) .. controls (608.5,819.32) and (529.47,800.94) .. (495.28,741.72) -- (609.89,675.55) -- cycle ; \draw   (724.5,609.38) .. controls (724.5,609.38) and (724.5,609.38) .. (724.5,609.38) .. controls (758.69,668.6) and (735.09,746.23) .. (671.8,782.77) .. controls (608.5,819.32) and (529.47,800.94) .. (495.28,741.72) ;
\draw  [draw opacity=0] (463.79,737.49) .. controls (429.6,796.71) and (350.57,815.09) .. (287.27,778.54) .. controls (223.97,742) and (200.38,664.37) .. (234.57,605.15) -- (349.18,671.32) -- cycle ; \draw   (463.79,737.49) .. controls (429.6,796.71) and (350.57,815.09) .. (287.27,778.54) .. controls (223.97,742) and (200.38,664.37) .. (234.57,605.15) ;
\draw  [draw opacity=0] (212.18,575.66) .. controls (143.8,575.66) and (88.36,516.41) .. (88.36,443.32) .. controls (88.36,370.23) and (143.8,310.97) .. (212.18,310.97) -- (212.18,443.32) -- cycle ; \draw   (212.18,575.66) .. controls (143.8,575.66) and (88.36,516.41) .. (88.36,443.32) .. controls (88.36,370.23) and (143.8,310.97) .. (212.18,310.97) ;
\draw    (498,182) -- (621.5,655.63) ;
\draw    (625,224) -- (498,699) ;
\draw    (711,325) -- (362,670) ;
\draw    (738,454) -- (262.78,582.61) ;
\draw    (698.13,583) -- (226,450) ;
\draw    (600,667) -- (248,326) ;
\draw    (467,699) -- (341.14,225.83) ;
\draw    (338.86,656.04) -- (463,178) ;
\draw    (253,562) -- (594,206) ;
\draw    (222,423) -- (697,298) ;
\draw    (263.02,302.61) -- (740,423) ;
\draw    (365,206) -- (712,560) ;
\draw   (418.15,178.47) .. controls (410.15,174.21) and (402.78,172.32) .. (396.01,172.82) .. controls (402.16,169.95) and (407.67,164.7) .. (412.58,157.08) ;
\draw   (558.59,182.5) .. controls (553.68,174.89) and (548.16,169.65) .. (542.01,166.79) .. controls (548.77,167.28) and (556.15,165.38) .. (564.14,161.11) ;
\draw   (678.59,258.16) .. controls (678.43,249.1) and (676.5,241.73) .. (672.79,236.06) .. controls (678.28,240.04) and (685.54,242.31) .. (694.59,242.9) ;
\draw   (295.45,246.33) .. controls (286.4,245.92) and (278.93,247.38) .. (273.03,250.72) .. controls (277.35,245.5) and (280.08,238.39) .. (281.25,229.4) ;
\draw   (224.32,358.26) .. controls (216.4,362.67) and (210.81,367.84) .. (207.56,373.79) .. controls (208.49,367.08) and (207.07,359.6) .. (203.32,351.34) ;
\draw   (217.91,499.86) .. controls (213.62,507.85) and (211.71,515.22) .. (212.17,521.98) .. controls (209.33,515.83) and (204.1,510.29) .. (196.5,505.36) ;
\draw   (281.17,621.47) .. controls (281.1,630.53) and (282.83,637.94) .. (286.4,643.71) .. controls (281.02,639.59) and (273.81,637.13) .. (264.79,636.31) ;
\draw   (400.68,697.66) .. controls (405.08,705.58) and (410.24,711.18) .. (416.18,714.44) .. controls (409.47,713.51) and (401.98,714.91) .. (393.72,718.64) ;
\draw   (542.37,701.91) .. controls (550.3,706.3) and (557.64,708.32) .. (564.41,707.94) .. controls (558.21,710.7) and (552.61,715.86) .. (547.58,723.39) ;
\draw   (661.81,640.45) .. controls (670.87,640.13) and (678.2,638.08) .. (683.81,634.27) .. controls (679.92,639.83) and (677.77,647.13) .. (677.33,656.19) ;
\draw   (737.44,522.13) .. controls (745.08,517.25) and (750.36,511.76) .. (753.25,505.63) .. controls (752.73,512.39) and (754.59,519.77) .. (758.81,527.79) ;
\draw   (744.03,383.43) .. controls (748.6,375.6) and (750.78,368.3) .. (750.55,361.53) .. controls (753.17,367.78) and (758.2,373.5) .. (765.62,378.7) ;
\draw   (470.69,39.24) .. controls (477.36,45.38) and (484.02,49.07) .. (490.69,50.3) .. controls (484.02,51.52) and (477.36,55.21) .. (470.69,61.35) ;
\draw   (676.1,89.48) .. controls (677.89,98.36) and (681.1,105.26) .. (685.78,110.17) .. controls (679.66,107.24) and (672.11,106.31) .. (663.11,107.36) ;
\draw   (834.23,633.03) .. controls (825.61,635.82) and (819.12,639.8) .. (814.77,645) .. controls (816.99,638.59) and (817.05,630.98) .. (814.99,622.15) ;
\draw   (884.86,431.43) .. controls (878.64,438.02) and (874.88,444.64) .. (873.57,451.29) .. controls (872.43,444.61) and (868.82,437.91) .. (862.76,431.17) ;
\draw   (268.21,101.76) .. controls (277.18,103.09) and (284.76,102.41) .. (290.97,99.69) .. controls (286.13,104.44) and (282.69,111.23) .. (280.61,120.05) ;
\draw   (129.26,249.01) .. controls (137.79,245.96) and (144.15,241.78) .. (148.34,236.44) .. controls (146.32,242.92) and (146.49,250.53) .. (148.82,259.29) ;
\draw   (77.17,455.75) .. controls (83.62,449.39) and (87.63,442.91) .. (89.17,436.31) .. controls (90.07,443.03) and (93.44,449.86) .. (99.25,456.81) ;
\draw   (136.33,651.55) .. controls (138.18,642.68) and (137.93,635.07) .. (135.56,628.71) .. controls (140.03,633.81) and (146.62,637.63) .. (155.3,640.21) ;
\draw   (490.94,840.99) .. controls (484.13,835.01) and (477.39,831.47) .. (470.69,830.4) .. controls (477.33,829.03) and (483.91,825.18) .. (490.43,818.9) ;
\draw   (823.79,230.17) .. controls (821.48,238.93) and (821.33,246.55) .. (823.36,253.01) .. controls (819.16,247.69) and (812.79,243.53) .. (804.25,240.5) ;
\draw   (686.92,786.6) .. controls (678.06,784.71) and (670.45,784.92) .. (664.08,787.26) .. controls (669.2,782.81) and (673.06,776.25) .. (675.68,767.57) ;
\draw   (291.83,793.74) .. controls (289.12,785.09) and (285.19,778.57) .. (280.03,774.18) .. controls (286.42,776.45) and (294.03,776.58) .. (302.87,774.6) ;
\draw   (521.45,225.24) .. controls (516.82,233.03) and (514.57,240.3) .. (514.74,247.08) .. controls (512.18,240.81) and (507.2,235.05) .. (499.82,229.78) ;
\draw   (622.9,274.39) .. controls (615.28,279.31) and (610.04,284.83) .. (607.18,290.97) .. controls (607.67,284.21) and (605.77,276.84) .. (601.51,268.84) ;
\draw   (684.56,367.24) .. controls (675.51,367.5) and (668.16,369.5) .. (662.53,373.27) .. controls (666.45,367.74) and (668.65,360.45) .. (669.15,351.4) ;
\draw   (697.25,476.31) .. controls (689.2,472.16) and (681.8,470.36) .. (675.05,470.94) .. controls (681.16,468) and (686.6,462.68) .. (691.41,454.99) ;
\draw   (648.95,580.82) .. controls (644.41,572.98) and (639.15,567.48) .. (633.14,564.32) .. controls (639.87,565.14) and (647.33,563.59) .. (655.52,559.72) ;
\draw   (559.66,644.16) .. controls (559.46,635.1) and (557.5,627.75) .. (553.76,622.09) .. controls (559.26,626.05) and (566.54,628.29) .. (575.59,628.84) ;
\draw   (443.97,659.38) .. controls (448.58,651.58) and (450.8,644.3) .. (450.6,637.52) .. controls (453.19,643.78) and (458.19,649.53) .. (465.59,654.76) ;
\draw   (339.54,610.31) .. controls (347.11,605.33) and (352.31,599.77) .. (355.12,593.6) .. controls (354.69,600.36) and (356.65,607.72) .. (360.98,615.68) ;
\draw   (277.79,521.5) .. controls (286.8,520.55) and (293.97,517.99) .. (299.3,513.79) .. controls (295.81,519.61) and (294.18,527.04) .. (294.38,536.1) ;
\draw   (262.05,401.4) .. controls (270.14,405.48) and (277.56,407.21) .. (284.31,406.56) .. controls (278.23,409.56) and (272.83,414.93) .. (268.1,422.66) ;
\draw   (308.73,301.64) .. controls (313.1,309.58) and (318.24,315.19) .. (324.17,318.48) .. controls (317.46,317.52) and (309.97,318.89) .. (301.69,322.59) ;
\draw   (403.95,230.4) .. controls (404.62,239.43) and (406.96,246.68) .. (410.99,252.13) .. controls (405.28,248.47) and (397.9,246.6) .. (388.84,246.52) ;

\draw (450,145) node [anchor=north west][inner sep=0.75pt]    {$\ 0$};
\draw (590,180) node [anchor=north west][inner sep=0.75pt]    {$\ 1$};
\draw (710,288) node [anchor=north west][inner sep=0.75pt]    {$2$};
\draw (732,425) node [anchor=north west][inner sep=0.75pt]    {$\ 3$};
\draw (450,702) node [anchor=north west][inner sep=0.75pt]    {$\ 6$};
\draw (710,567) node [anchor=north west][inner sep=0.75pt]    {$4$};
\draw (608,665) node [anchor=north west][inner sep=0.75pt]    {$5$};
\draw (300,180) node [anchor=north west][inner sep=0.75pt]  [rotate=-359.64]  {$\ 11$};
\draw (208,285) node [anchor=north west][inner sep=0.75pt]  [rotate=-359.64]  {$10$};
\draw (175,423) node [anchor=north west][inner sep=0.75pt]  [rotate=-359.64]  {$\ 9$};
\draw (220,562) node [anchor=north west][inner sep=0.75pt]  [rotate=-359.64]  {$8$};
\draw (325,665) node [anchor=north west][inner sep=0.75pt]  [rotate=-359.64]  {$7$};
\end{tikzpicture}

\caption{$S=\{ 2,5,11\}$}\label{a}
\end{subfigure}
\hfill
\begin{subfigure}{0.45\textwidth}
\tikzset{every picture/.style={line width=0.75pt}} 
\begin{tikzpicture}[x=0.23pt,y=0.23pt,yscale=-1,xscale=1]
\draw   (200,440.5) .. controls (200,285.58) and (325.58,160) .. (480.5,160) .. controls (635.42,160) and (761,285.58) .. (761,440.5) .. controls (761,595.42) and (635.42,721) .. (480.5,721) .. controls (325.58,721) and (200,595.42) .. (200,440.5) -- cycle ;
\draw  [fill={rgb, 255:red, 255; green, 255; blue, 255 }  ,fill opacity=1 ] (454.13,161) .. controls (454.13,146.43) and (465.93,134.63) .. (480.5,134.63) .. controls (495.07,134.63) and (506.88,146.43) .. (506.88,161) .. controls (506.88,175.57) and (495.07,187.38) .. (480.5,187.38) .. controls (465.93,187.38) and (454.13,175.57) .. (454.13,161) -- cycle ;
\draw  [fill={rgb, 255:red, 255; green, 255; blue, 255 }  ,fill opacity=1 ] (593.13,199) .. controls (593.13,184.43) and (604.93,172.63) .. (619.5,172.63) .. controls (634.07,172.63) and (645.88,184.43) .. (645.88,199) .. controls (645.88,213.57) and (634.07,225.38) .. (619.5,225.38) .. controls (604.93,225.38) and (593.13,213.57) .. (593.13,199) -- cycle ;
\draw  [fill={rgb, 255:red, 255; green, 255; blue, 255 }  ,fill opacity=1 ] (698.13,303) .. controls (698.13,288.43) and (709.93,276.63) .. (724.5,276.63) .. controls (739.07,276.63) and (750.88,288.43) .. (750.88,303) .. controls (750.88,317.57) and (739.07,329.38) .. (724.5,329.38) .. controls (709.93,329.38) and (698.13,317.57) .. (698.13,303) -- cycle ;
\draw  [fill={rgb, 255:red, 255; green, 255; blue, 255 }  ,fill opacity=1 ] (733.13,440) .. controls (733.13,425.43) and (744.93,413.63) .. (759.5,413.63) .. controls (774.07,413.63) and (785.88,425.43) .. (785.88,440) .. controls (785.88,454.57) and (774.07,466.38) .. (759.5,466.38) .. controls (744.93,466.38) and (733.13,454.57) .. (733.13,440) -- cycle ;
\draw  [fill={rgb, 255:red, 255; green, 255; blue, 255 }  ,fill opacity=1 ] (455.13,719) .. controls (455.13,704.43) and (466.93,692.63) .. (481.5,692.63) .. controls (496.07,692.63) and (507.88,704.43) .. (507.88,719) .. controls (507.88,733.57) and (496.07,745.38) .. (481.5,745.38) .. controls (466.93,745.38) and (455.13,733.57) .. (455.13,719) -- cycle ;
\draw  [fill={rgb, 255:red, 255; green, 255; blue, 255 }  ,fill opacity=1 ] (698.13,583) .. controls (698.13,568.43) and (709.93,556.63) .. (724.5,556.63) .. controls (739.07,556.63) and (750.88,568.43) .. (750.88,583) .. controls (750.88,597.57) and (739.07,609.38) .. (724.5,609.38) .. controls (709.93,609.38) and (698.13,597.57) .. (698.13,583) -- cycle ;
\draw  [fill={rgb, 255:red, 255; green, 255; blue, 255 }  ,fill opacity=1 ] (595.13,682) .. controls (595.13,667.43) and (606.93,655.63) .. (621.5,655.63) .. controls (636.07,655.63) and (647.88,667.43) .. (647.88,682) .. controls (647.88,696.57) and (636.07,708.38) .. (621.5,708.38) .. controls (606.93,708.38) and (595.13,696.57) .. (595.13,682) -- cycle ;
\draw  [fill={rgb, 255:red, 255; green, 255; blue, 255 }  ,fill opacity=1 ] (314.6,199.62) .. controls (314.51,185.05) and (326.25,173.17) .. (340.81,173.08) .. controls (355.38,172.98) and (367.26,184.72) .. (367.35,199.29) .. controls (367.44,213.85) and (355.71,225.73) .. (341.14,225.83) .. controls (326.58,225.92) and (314.69,214.18) .. (314.6,199.62) -- cycle ;
\draw  [fill={rgb, 255:red, 255; green, 255; blue, 255 }  ,fill opacity=1 ] (210.27,302.94) .. controls (210.18,288.38) and (221.91,276.49) .. (236.48,276.4) .. controls (251.04,276.31) and (262.92,288.04) .. (263.02,302.61) .. controls (263.11,317.18) and (251.37,329.06) .. (236.81,329.15) .. controls (222.24,329.24) and (210.36,317.51) .. (210.27,302.94) -- cycle ;
\draw  [fill={rgb, 255:red, 255; green, 255; blue, 255 }  ,fill opacity=1 ] (175.13,439.72) .. controls (175.04,425.15) and (186.77,413.27) .. (201.34,413.18) .. controls (215.9,413.09) and (227.78,424.82) .. (227.88,439.39) .. controls (227.97,453.95) and (216.23,465.84) .. (201.67,465.93) .. controls (187.1,466.02) and (175.22,454.29) .. (175.13,439.72) -- cycle ;
\draw  [fill={rgb, 255:red, 255; green, 255; blue, 255 }  ,fill opacity=1 ] (210.03,582.94) .. controls (209.93,568.37) and (221.67,556.49) .. (236.24,556.4) .. controls (250.8,556.3) and (262.68,568.04) .. (262.78,582.61) .. controls (262.87,597.17) and (251.13,609.05) .. (236.57,609.15) .. controls (222,609.24) and (210.12,597.5) .. (210.03,582.94) -- cycle ;
\draw  [fill={rgb, 255:red, 255; green, 255; blue, 255 }  ,fill opacity=1 ] (312.65,682.58) .. controls (312.56,668.02) and (324.29,656.13) .. (338.86,656.04) .. controls (353.43,655.95) and (365.31,667.68) .. (365.4,682.25) .. controls (365.49,696.82) and (353.76,708.7) .. (339.19,708.79) .. controls (324.62,708.88) and (312.74,697.15) .. (312.65,682.58) -- cycle ;
\draw  [draw opacity=0] (184.52,421.19) .. controls (184.52,421.19) and (184.52,421.19) .. (184.52,421.19) .. controls (125.3,386.99) and (106.92,307.96) .. (143.46,244.67) .. controls (180.01,181.37) and (257.64,157.77) .. (316.86,191.96) -- (250.69,306.57) -- cycle ; \draw   (184.52,421.19) .. controls (184.52,421.19) and (184.52,421.19) .. (184.52,421.19) .. controls (125.3,386.99) and (106.92,307.96) .. (143.46,244.67) .. controls (180.01,181.37) and (257.64,157.77) .. (316.86,191.96) ;
\draw  [draw opacity=0] (347.64,174.46) .. controls (347.64,174.46) and (347.64,174.46) .. (347.64,174.46) .. controls (347.64,106.08) and (406.89,50.65) .. (479.98,50.65) .. controls (553.07,50.65) and (612.33,106.08) .. (612.33,174.46) -- (479.98,174.46) -- cycle ; \draw   (347.64,174.46) .. controls (347.64,174.46) and (347.64,174.46) .. (347.64,174.46) .. controls (347.64,106.08) and (406.89,50.65) .. (479.98,50.65) .. controls (553.07,50.65) and (612.33,106.08) .. (612.33,174.46) ;
\draw  [draw opacity=0] (645.81,192.85) .. controls (645.81,192.85) and (645.81,192.85) .. (645.81,192.85) .. controls (705.03,158.66) and (782.67,182.26) .. (819.21,245.56) .. controls (855.76,308.85) and (837.38,387.89) .. (778.15,422.08) -- (711.98,307.46) -- cycle ; \draw   (645.81,192.85) .. controls (645.81,192.85) and (645.81,192.85) .. (645.81,192.85) .. controls (705.03,158.66) and (782.67,182.26) .. (819.21,245.56) .. controls (855.76,308.85) and (837.38,387.89) .. (778.15,422.08) ;
\draw  [draw opacity=0] (779.15,459.85) .. controls (779.15,459.85) and (779.15,459.85) .. (779.15,459.85) .. controls (838.38,494.04) and (856.76,573.07) .. (820.21,636.37) .. controls (783.67,699.67) and (706.03,723.27) .. (646.81,689.08) -- (712.98,574.46) -- cycle ; \draw   (779.15,459.85) .. controls (779.15,459.85) and (779.15,459.85) .. (779.15,459.85) .. controls (838.38,494.04) and (856.76,573.07) .. (820.21,636.37) .. controls (783.67,699.67) and (706.03,723.27) .. (646.81,689.08) ;
\draw  [draw opacity=0] (613.33,706.46) .. controls (613.33,706.46) and (613.33,706.46) .. (613.33,706.46) .. controls (613.33,774.85) and (554.07,830.28) .. (480.98,830.28) .. controls (407.89,830.28) and (348.64,774.85) .. (348.64,706.46) -- (480.98,706.46) -- cycle ; \draw   (613.33,706.46) .. controls (613.33,706.46) and (613.33,706.46) .. (613.33,706.46) .. controls (613.33,774.85) and (554.07,830.28) .. (480.98,830.28) .. controls (407.89,830.28) and (348.64,774.85) .. (348.64,706.46) ;
\draw  [draw opacity=0] (313.15,687.8) .. controls (253.93,721.99) and (176.3,698.39) .. (139.76,635.09) .. controls (103.21,571.79) and (121.59,492.76) .. (180.81,458.57) -- (246.98,573.18) -- cycle ; \draw   (313.15,687.8) .. controls (253.93,721.99) and (176.3,698.39) .. (139.76,635.09) .. controls (103.21,571.79) and (121.59,492.76) .. (180.81,458.57) ;
\draw  [draw opacity=0] (231.08,277.74) .. controls (196.89,218.52) and (220.48,140.89) .. (283.78,104.35) .. controls (347.08,67.8) and (426.11,86.18) .. (460.3,145.4) -- (345.69,211.57) -- cycle ; \draw   (231.08,277.74) .. controls (196.89,218.52) and (220.48,140.89) .. (283.78,104.35) .. controls (347.08,67.8) and (426.11,86.18) .. (460.3,145.4) ;
\draw  [draw opacity=0] (502.08,146.4) .. controls (502.08,146.4) and (502.08,146.4) .. (502.08,146.4) .. controls (536.27,87.18) and (615.3,68.8) .. (678.6,105.35) .. controls (741.9,141.89) and (765.49,219.52) .. (731.3,278.74) -- (616.69,212.57) -- cycle ; \draw   (502.08,146.4) .. controls (502.08,146.4) and (502.08,146.4) .. (502.08,146.4) .. controls (536.27,87.18) and (615.3,68.8) .. (678.6,105.35) .. controls (741.9,141.89) and (765.49,219.52) .. (731.3,278.74) ;
\draw  [draw opacity=0] (749.69,310.23) .. controls (749.69,310.23) and (749.69,310.23) .. (749.69,310.23) .. controls (818.07,310.23) and (873.5,369.48) .. (873.5,442.57) .. controls (873.5,515.66) and (818.07,574.92) .. (749.69,574.92) -- (749.69,442.57) -- cycle ; \draw   (749.69,310.23) .. controls (749.69,310.23) and (749.69,310.23) .. (749.69,310.23) .. controls (818.07,310.23) and (873.5,369.48) .. (873.5,442.57) .. controls (873.5,515.66) and (818.07,574.92) .. (749.69,574.92) ;
\draw  [draw opacity=0] (724.5,609.38) .. controls (724.5,609.38) and (724.5,609.38) .. (724.5,609.38) .. controls (758.69,668.6) and (735.09,746.23) .. (671.8,782.77) .. controls (608.5,819.32) and (529.47,800.94) .. (495.28,741.72) -- (609.89,675.55) -- cycle ; \draw   (724.5,609.38) .. controls (724.5,609.38) and (724.5,609.38) .. (724.5,609.38) .. controls (758.69,668.6) and (735.09,746.23) .. (671.8,782.77) .. controls (608.5,819.32) and (529.47,800.94) .. (495.28,741.72) ;
\draw  [draw opacity=0] (463.79,737.49) .. controls (429.6,796.71) and (350.57,815.09) .. (287.27,778.54) .. controls (223.97,742) and (200.38,664.37) .. (234.57,605.15) -- (349.18,671.32) -- cycle ; \draw   (463.79,737.49) .. controls (429.6,796.71) and (350.57,815.09) .. (287.27,778.54) .. controls (223.97,742) and (200.38,664.37) .. (234.57,605.15) ;
\draw  [draw opacity=0] (212.18,575.66) .. controls (143.8,575.66) and (88.36,516.41) .. (88.36,443.32) .. controls (88.36,370.23) and (143.8,310.97) .. (212.18,310.97) -- (212.18,443.32) -- cycle ; \draw   (212.18,575.66) .. controls (143.8,575.66) and (88.36,516.41) .. (88.36,443.32) .. controls (88.36,370.23) and (143.8,310.97) .. (212.18,310.97) ;
\draw    (498,182) -- (621.5,655.63) ;
\draw    (625,224) -- (498,699) ;
\draw    (711,325) -- (362,670) ;
\draw    (738,454) -- (262.78,582.61) ;
\draw    (698.13,583) -- (226,450) ;
\draw    (600,667) -- (248,326) ;
\draw    (467,699) -- (341.14,225.83) ;
\draw    (338.86,656.04) -- (463,178) ;
\draw    (253,562) -- (594,206) ;
\draw    (222,423) -- (697,298) ;
\draw    (263.02,302.61) -- (740,423) ;
\draw    (365,206) -- (712,560) ;
\draw   (418.15,178.47) .. controls (410.15,174.21) and (402.78,172.32) .. (396.01,172.82) .. controls (402.16,169.95) and (407.67,164.7) .. (412.58,157.08) ;
\draw   (558.59,182.5) .. controls (553.68,174.89) and (548.16,169.65) .. (542.01,166.79) .. controls (548.77,167.28) and (556.15,165.38) .. (564.14,161.11) ;
\draw   (678.59,258.16) .. controls (678.43,249.1) and (676.5,241.73) .. (672.79,236.06) .. controls (678.28,240.04) and (685.54,242.31) .. (694.59,242.9) ;
\draw   (295.45,246.33) .. controls (286.4,245.92) and (278.93,247.38) .. (273.03,250.72) .. controls (277.35,245.5) and (280.08,238.39) .. (281.25,229.4) ;
\draw   (224.32,358.26) .. controls (216.4,362.67) and (210.81,367.84) .. (207.56,373.79) .. controls (208.49,367.08) and (207.07,359.6) .. (203.32,351.34) ;
\draw   (217.91,499.86) .. controls (213.62,507.85) and (211.71,515.22) .. (212.17,521.98) .. controls (209.33,515.83) and (204.1,510.29) .. (196.5,505.36) ;
\draw   (281.17,621.47) .. controls (281.1,630.53) and (282.83,637.94) .. (286.4,643.71) .. controls (281.02,639.59) and (273.81,637.13) .. (264.79,636.31) ;
\draw   (400.68,697.66) .. controls (405.08,705.58) and (410.24,711.18) .. (416.18,714.44) .. controls (409.47,713.51) and (401.98,714.91) .. (393.72,718.64) ;
\draw   (542.37,701.91) .. controls (550.3,706.3) and (557.64,708.32) .. (564.41,707.94) .. controls (558.21,710.7) and (552.61,715.86) .. (547.58,723.39) ;
\draw   (661.81,640.45) .. controls (670.87,640.13) and (678.2,638.08) .. (683.81,634.27) .. controls (679.92,639.83) and (677.77,647.13) .. (677.33,656.19) ;
\draw   (737.44,522.13) .. controls (745.08,517.25) and (750.36,511.76) .. (753.25,505.63) .. controls (752.73,512.39) and (754.59,519.77) .. (758.81,527.79) ;
\draw   (744.03,383.43) .. controls (748.6,375.6) and (750.78,368.3) .. (750.55,361.53) .. controls (753.17,367.78) and (758.2,373.5) .. (765.62,378.7) ;
\draw   (521.45,225.24) .. controls (516.82,233.03) and (514.57,240.3) .. (514.74,247.08) .. controls (512.18,240.81) and (507.2,235.05) .. (499.82,229.78) ;
\draw   (622.9,274.39) .. controls (615.28,279.31) and (610.04,284.83) .. (607.18,290.97) .. controls (607.67,284.21) and (605.77,276.84) .. (601.51,268.84) ;
\draw   (684.56,367.24) .. controls (675.51,367.5) and (668.16,369.5) .. (662.53,373.27) .. controls (666.45,367.74) and (668.65,360.45) .. (669.15,351.4) ;
\draw   (697.25,476.31) .. controls (689.2,472.16) and (681.8,470.36) .. (675.05,470.94) .. controls (681.16,468) and (686.6,462.68) .. (691.41,454.99) ;
\draw   (648.95,580.82) .. controls (644.41,572.98) and (639.15,567.48) .. (633.14,564.32) .. controls (639.87,565.14) and (647.33,563.59) .. (655.52,559.72) ;
\draw   (559.66,644.16) .. controls (559.46,635.1) and (557.5,627.75) .. (553.76,622.09) .. controls (559.26,626.05) and (566.54,628.29) .. (575.59,628.84) ;
\draw   (443.97,659.38) .. controls (448.58,651.58) and (450.8,644.3) .. (450.6,637.52) .. controls (453.19,643.78) and (458.19,649.53) .. (465.59,654.76) ;
\draw   (339.54,610.31) .. controls (347.11,605.33) and (352.31,599.77) .. (355.12,593.6) .. controls (354.69,600.36) and (356.65,607.72) .. (360.98,615.68) ;
\draw   (277.79,521.5) .. controls (286.8,520.55) and (293.97,517.99) .. (299.3,513.79) .. controls (295.81,519.61) and (294.18,527.04) .. (294.38,536.1) ;
\draw   (262.05,401.4) .. controls (270.14,405.48) and (277.56,407.21) .. (284.31,406.56) .. controls (278.23,409.56) and (272.83,414.93) .. (268.1,422.66) ;
\draw   (308.73,301.64) .. controls (313.1,309.58) and (318.24,315.19) .. (324.17,318.48) .. controls (317.46,317.52) and (309.97,318.89) .. (301.69,322.59) ;
\draw   (403.95,230.4) .. controls (404.62,239.43) and (406.96,246.68) .. (410.99,252.13) .. controls (405.28,248.47) and (397.9,246.6) .. (388.84,246.52) ;

\draw (450,145) node [anchor=north west][inner sep=0.75pt]    {$\ 0$};
\draw (590,180) node [anchor=north west][inner sep=0.75pt]    {$\ 1$};
\draw (710,288) node [anchor=north west][inner sep=0.75pt]    {$2$};
\draw (732,425) node [anchor=north west][inner sep=0.75pt]    {$\ 3$};
\draw (450,702) node [anchor=north west][inner sep=0.75pt]    {$\ 6$};
\draw (710,567) node [anchor=north west][inner sep=0.75pt]    {$4$};
\draw (608,665) node [anchor=north west][inner sep=0.75pt]    {$5$};
\draw (300,180) node [anchor=north west][inner sep=0.75pt]  [rotate=-359.64]  {$\ 11$};
\draw (208,285) node [anchor=north west][inner sep=0.75pt]  [rotate=-359.64]  {$10$};
\draw (175,423) node [anchor=north west][inner sep=0.75pt]  [rotate=-359.64]  {$\ 9$};
\draw (220,562) node [anchor=north west][inner sep=0.75pt]  [rotate=-359.64]  {$8$};
\draw (325,665) node [anchor=north west][inner sep=0.75pt]  [rotate=-359.64]  {$7$};
\end{tikzpicture}
\caption{$S=\{ 2,5,10,11\}$} \label{b}
\end{subfigure}
\caption{The graph $Circ(\mathbb{Z}_{12}, S)$.}\label{main}
\end{figure}
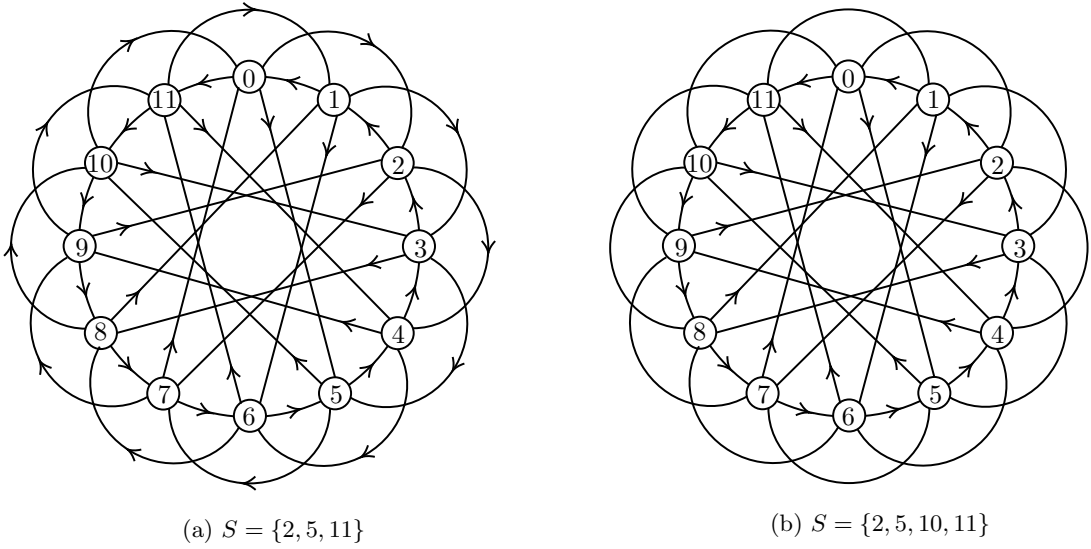

\begin{lema}\label{CharaNewIntegSum} Let $S$ be a skew-symmetric subset of $\mathbb{Z}_n$ and $t (\neq 0)\in \mathbb{Q}$. If $\sum\limits_{k\in S} it\sqrt{3}(\omega_n^{jk}-\omega_n^{-jk})$ is an integer for each $j\in \{0,...,n-1\}$ then 
\begin{equation*}
			S= \left\{ \begin{array}{ll}
				\emptyset & \text{ if } n\not\equiv 0\Mod 3 \\ 
				\bigcup\limits_{d\in \mathscr{D}}S_n(d) & \text{ if } n\equiv 0\Mod 3,
			\end{array}\right.
		\end{equation*}
		where $\mathscr{D} \subseteq \{ d: d\mid \frac{n}{3}\}$ and $S_n(d)\in \{ G_{n,3}^{1}(d) , G_{n,3}^{2}(d)\}$. 
	\end{lema}
	\begin{proof}
 Let $ \alpha_j = \sum\limits_{k\in S} it\sqrt{3}(\omega_n^{jk}-\omega_n^{-jk})$ for each $j\in \{0,...,n-1\}$. Assume that $n\not\equiv 0 \Mod 3$, so that $\Phi_n(x)$ is irreducible in $\mathbb{Q}(\omega_3)[x]$. Then $\omega_n^j$ is a root of  $p(x)= \sum\limits_{k\in S} it\sqrt{3}( x^{k} - x^{n-k})- \alpha_j \in \mathbb{Q}(\omega_3)[x]$. Therefore $p(x)$ is a multiple of the irreducible polynomial $\Phi_n(x)$, and so $\omega_n^{-j}=\omega_n^{n-j}$ is also a root of $p(x)$, that is, $\alpha_j=\alpha_{n-j}$. Thus $\alpha_j=-\alpha_{n-j}$ and $\alpha_j=\alpha_{n-j}$ implies that $\alpha_j=\alpha_{n-j}=0$. By Lemma \ref{Sqrt3ZeroSum}, we get $S=\emptyset$.  
			
Assume that $n\equiv 0 \Mod 3$. Let $v$ be the vector of length $n$ defined by
			$$v_k= \left\{ \begin{array}{rl}
				it\sqrt{3} & \mbox{if }  k\in S \\
				-it\sqrt{3} & \mbox{if }  n-k\in S\\ 
				0 &  \mbox{otherwise}. 
			\end{array}\right.$$ 
Since $v\in \mathbb{Q}^n(\omega_3)$ and the $j$-th coordinate of $Ev$ is $\alpha_j$, we have $Ev \in \mathbb{Q}^n$. Thus $v$ satisfies all the conditions of Lemma~\ref{MainLemmNeceCond}. Hence 
\begin{equation*}
			S= \left\{ \begin{array}{ll}
				\emptyset & \text{ if } n\not\equiv 0\Mod 3 \\ 
				\bigcup\limits_{d\in \mathscr{D}}S_n(d) & \text{ if } n\equiv 0\Mod 3,
			\end{array}\right.
		\end{equation*}
where $\mathscr{D} \subseteq \{ d: d\mid \frac{n}{3}\}$ and $S_n(d)\in \{ G_{n,3}^{1}(d) , G_{n,3}^{2}(d)\}$. 
\end{proof}
	
\begin{lema}\label{Sqrt3NecessIntSum} Let $S$ be a skew-symmetric subset of $\mathbb{Z}_n$ and $t (\neq 0)\in \mathbb{Q}$. If $\sum\limits_{k\in S}it\sqrt{3}(\omega_n^{jk}-\omega_n^{-jk})$ is an integer for each $j\in \{0,...,n-1\}$ then $\sum\limits_{k\in S\cup S^{-1}} \omega_n^{jk} \in \mathbb{Z}$  for each $j\in \{0,...,n-1\}$.
	\end{lema}
	\begin{proof}
	By Lemma \ref{CharaNewIntegSum}, there is nothing to prove if $n\not\equiv 0 \Mod 3$. Assume  $n\equiv 0 \Mod 3$. Again by Lemma \ref{CharaNewIntegSum},  $S= \bigcup\limits_{d\in \mathscr{D}}S_n(d)$, where $\mathscr{D} \subseteq \{ d: d\mid \frac{n}{3}\}$ and $S_n(d)\in \{ G_{n,3}^{1}(d) , G_{n,3}^{2}(d)\}$. Therefore $S\cup S^{-1}= \bigcup\limits_{d\in \mathscr{D}}G_n(d)$. By Theorem \ref{2006integral}, the graph $\text{Circ}(\mathbb{Z}_n, S \cup S^{-1})$ is integral. Note that $\sum\limits_{k\in S\cup S^{-1}} \omega_n^{jk}$ is an eigenvalue of $\text{Circ}(\mathbb{Z}_n, S \cup S^{-1})$ for each $j\in \{0,...,n-1\}$. Hence $\sum\limits_{k\in S\cup S^{-1}} \omega_n^{jk} \in \mathbb{Z}$  for each $j\in \{0,...,n-1\}$.
	\end{proof}

	\begin{lema}\label{SeperatIntegMixedGraph}
Let $S\subseteq \mathbb{Z}_n$ such that $0\notin S$. Then the mixed circulant graph $\text{Circ}(\mathbb{Z}_n,S)$ is HS-integral if and only if both $\text{Circ}(\mathbb{Z}_n,{S\setminus \overline{S}})$ and $\text{Circ}(\mathbb{Z}_n, {\overline{S}})$ are HS-integral.
	\end{lema}
	\begin{proof}
	Assume that the mixed circulant graph $\text{Circ}(\mathbb{Z}_n,S)$ is HS-integral. Let the HS-spectrum of $\text{Circ}(\mathbb{Z}_n,S)$ be $\{\gamma_0,\gamma_1,...,\gamma_{n-1} \}$, where $\gamma_j=\lambda_j+\mu_j$, 
	$$\lambda_j= \sum\limits_{k\in S\setminus \overline{S}} \omega_n^{jk} \text{ and } \mu_j=  \sum\limits_{k\in \overline{S}} (\omega_6\omega_n^{jk} + \omega_6^5\omega_n^{-jk}) \textnormal{ for each } j\in \{0,...,n-1\}.$$
By assumption $\gamma_j \in \mathbb{Z}$, and so $ \gamma_j - \gamma_{n-j}= \sum\limits_{k\in \overline{S}} i\sqrt{3}(\omega_n^{jk}-\omega_n^{-jk})  \in \mathbb{Z}$ for each $j\in \{0,...,n-1\}$. By Lemma \ref{Sqrt3NecessIntSum}, we get $\sum\limits_{ k \in \overline{S} \cup \overline{S}^{-1}}\omega_n^{jk} \in \mathbb{Z}$ for each $j\in \{0,...,n-1\}$. 
	 Since 
			\begin{align*}
	\mu_j= \frac{1}{2} \sum\limits_{ k \in \overline{S} \cup \overline{S}^{-1}}\omega_n^{jk} + \frac{1}{2} \sum\limits_{k\in \overline{S}} i\sqrt{3}(\omega_n^{jk}-\omega_n^{-jk}),
			\end{align*} 
the number $\mu_j$ is a rational algebraic integer for each $j\in \{0,...,n-1\}$. Thus $\text{Circ}(\mathbb{Z}_n,\overline{S})$ is HS-integral. Now we have $\gamma_j, \mu_j \in \mathbb{Z}$ for each $j\in \{0,...,n-1\}$, and so $\lambda_j = \gamma_j -\mu_j \in \mathbb{Z}$ for each $j\in \{0,...,n-1\}$. Thus $\text{Circ}(\mathbb{Z}_n,S\setminus \overline{S})$ is also HS-integral.
			
			Conversely, assume that both $\text{Circ}(\mathbb{Z}_n,S\setminus \overline{S})$ and $\text{Circ}(\mathbb{Z}_n, \overline{S})$ are HS-integral. Then Lemma \ref{SpecMixCayGraph} implies that $\text{Circ}(\mathbb{Z}_n,S)$ is HS-integral.
	\end{proof}

\begin{theorem}\label{CharaOfIntegMixedGraph}
Let $S\subseteq \mathbb{Z}_n$ such that $0\notin S$. Then the mixed circulant graph $\text{Circ}(\mathbb{Z}_n,S)$ is  HS-integral if and only if  $S \setminus \overline{S} = \bigcup\limits_{d\in \mathscr{D}_1}G_n(d)$ and
		\begin{equation*}
			\overline{S}= \left\{ \begin{array}{ll}
				\emptyset & \text{ if } n\not\equiv 0\Mod 3 \\ 
				\bigcup\limits_{d\in \mathscr{D}_2}S_n(d) & \text{ if } n\equiv 0\Mod 3,
			\end{array}\right.
		\end{equation*}
		where $\mathscr{D}_1 \subseteq \{ d: d\mid n\}$, $\mathscr{D}_2 \subseteq \{ d: d\mid \frac{n}{3}\}$, $\mathscr{D}_1 \cap \mathscr{D}_2 = \emptyset$ and $S_n(d)\in \{ G_{n,3}^{1}(d) , G_{n,3}^{2}(d)\}$.
	\end{theorem}
	\begin{proof} By Lemma~\ref{SeperatIntegMixedGraph}, $\text{Circ}(\mathbb{Z}_n,S)$ is HS-integral if and only if $\text{Circ}(\mathbb{Z}_n,S\setminus \overline{S})$ and $\text{Circ}(\mathbb{Z}_n, \overline{S})$ are HS-integral. Therefore the result follows from Theorem~\ref{2006integral} and Theorem~\ref{CharaOfOrieCayGraph}.
	\end{proof}
	
	The following example illustrates Theorem~\ref{CharaOfIntegMixedGraph}.

\begin{ex}\label{ex2}
Consider $\Gamma= \mathbb{Z}_{12}$ and $S=\{ 2,5,10,11\}$. The mixed graph $\text{Circ}(\mathbb{Z}_{12}, S)$ is shown in Figure~\ref{b}. We see that $G_{12,3}^2(1)=\{5,11\}$ and $G_{12}(2)=\{2,10\}$. Therefore $S=G_{12,3}^2(1) \cup G_{12}(2)$, and hence $\text{Circ}(\mathbb{Z}_{12}, S)$ is HS-integral. Further, using Corollary~\ref{SpecMixCayGraph}, the HS-eigenvalues of  $\text{Circ}(\mathbb{Z}_{12}, S)$ are obtained as
$$\gamma_j=2\cos\left(\frac{\pi j}{3}\right) + \cos\left(\frac{5\pi j}{6}\right) + \cos\left(\frac{11\pi j}{6}\right) - \sqrt{3} \left[ \sin\left(\frac{5\pi j}{6}\right) + \sin\left(\frac{11\pi j}{6}\right) \right],$$
$\text{for each } j \in \mathbb{Z}_{12}$. One can see that $\gamma_0=4, \gamma_1=1, \gamma_2=3, \gamma_3=-2, \gamma_4=1, \gamma_5=1, \gamma_6=0, \gamma_7=1,$ $ \gamma_8=-5, \gamma_9=-2, \gamma_{10}=-3$ and $\gamma_{11}=1$. Thus all the HS-eigenvalues of $\text{Circ}(\mathbb{Z}_{12}, S)$ are integers. \qed
\end{ex}

	
\section{Characterization of Eisenstein integral mixed circulant graphs}\label{Eisenstein}

In 1918, Ramanujan \cite{ramanujan1918certain} introduced the sum, known as Ramanujan sum, defined by
	\begin{equation}\label{ramasum}
		\begin{split}
			C_n(q)= \sum_{a\in G_n(1)} \cos \left(\frac{2\pi a q}{n}\right)~~\text{for each }n,q \in \mathbb{N}.
		\end{split} 
	\end{equation} 

It is well known that $C_n(q)$ is an integer for all $n,q \in \mathbb{Z}$. For $n \equiv 0 \Mod 3$, define 
$$T_n(q)=\sum\limits_{a\in G_{n,3}^1(1)}i\sqrt{3}(\omega_n^{aq}- \omega_n^{-aq}) = \sum\limits_{a\in G_{n,3}^1(1)} -2 \sqrt{3} \sin \bigg( \frac{2\pi aq}{n} \bigg).$$

\begin{lema}\label{Tn(q)IsIntegerForAll}
Let $n \equiv 0 \Mod 3$. Then $T_n(q) \in \mathbb{Z}$ for all $n,q \in \mathbb{Z}$.
\end{lema}
\begin{proof}
We have
\begin{equation*}
			\begin{split}
Z_n^1(q) 	= \sum\limits_{a\in G_{n,3}^1(1)}\left(\omega_3\omega_n^{aq} + \omega_3^2 \omega_n^{-aq}\right)
			=-\frac{1}{2} \sum\limits_{q\in G_{n}(1)}\omega_n^{jq} + \frac{i\sqrt{3}}{2} \sum\limits_{a\in G_{n,3}^1(1)} (\omega_n^{aq} - \omega_n^{-aq})
			= -\frac{C_n(q)}{2} + \frac{T_n(q)}{2}.\\
			\end{split} 
	\end{equation*}
Thus $T_n(q)=  2 Z_n^1(q) + C_n(q) \in \mathbb{Z}$ for all $n,q \in \mathbb{Z}$.
\end{proof}
	
\begin{lema}\label{NewLemmaEquivaTrans1} Let $n=3^tm$ with $m\not\equiv 0 \Mod 3$. Then the following statements hold.
\begin{enumerate}[label=(\roman*)]
\item If $t=1$ then $G_n(1)=\left( m+3G_{\frac{n}{3}}(1) \right) \cup \left( 2m+3G_{\frac{n}{3}}(1) \right)$.
\item If $t=1$ then $$G_{n,3}^1(1) = \left\{ \begin{array}{ll}
			m+3G_{\frac{n}{3}}(1)     & \mbox{if } m \equiv 1 \Mod 3  \\
			2m+3G_{\frac{n}{3}}(1) & \mbox{if } m \equiv 2 \Mod 3.
		\end{array}\right. $$
\item If $t \geq 2$ then $$G_{n,3}^1(1) = \left\{ \begin{array}{ll}
			\left(m+3G_{\frac{n}{3}}(1)\right)  \cup \left(4m+3G_{\frac{n}{3},3}^2(1)\right)   & \mbox{if } m \equiv 1 \Mod 3  \\
			 \left(2m+3G_{\frac{n}{3}}(1)\right) \cup \left(5m+3G_{\frac{n}{3},3}^1(1)\right) & \mbox{if } m \equiv 2 \Mod 3.
		\end{array}\right. $$
\item If $t \geq 2$ then $$G_{n,3}^1(1) = \left\{ \begin{array}{ll}
			\left(7m+3G_{\frac{n}{3}}(1)\right) \cup \left(4m+3G_{\frac{n}{3},3}^1(1)\right)  & \mbox{if } m \equiv 1 \Mod 3  \\
			 \left(8m+3G_{\frac{n}{3}}(1)\right) \cup \left(5m+3G_{\frac{n}{3},3}^2(1)\right) & \mbox{if } m \equiv 2 \Mod 3.
		\end{array}\right. $$
\item If $t \geq 2$ then $$G_{n,3}^1(1) = \left\{ \begin{array}{ll}
			\left(m+3G_{\frac{n}{3}}(1)\right) \cup \left(4m+3G_{\frac{n}{3}}(1)\right) \cup \left(7m+3G_{\frac{n}{3}}(1)\right)   & \mbox{if } m \equiv 1 \Mod 3  \\
			\left(2m+3G_{\frac{n}{3}}(1)\right) \cup \left(5m+3G_{\frac{n}{3}}(1)\right) \cup \left(8m+3G_{\frac{n}{3}}(1)\right) & \mbox{if } m \equiv 2 \Mod 3.
		\end{array}\right. $$
\end{enumerate}
\end{lema}
\begin{proof}
\begin{enumerate}[label=(\roman*)]
\item Assume that $t=1$. Let $k \in m+3G_{\frac{n}{3}}(1)$. We get $k=m+3r$ for some $r\in G_{\frac{n}{3}}(1)$. Then $\gcd(r, \frac{n}{3})=1$ suggests that $\gcd(m+3r, n)=1$. Therefore $m+3G_{\frac{n}{3}}(1) \subseteq G_n(1)$. Similarly, $2m+3G_{\frac{n}{3}}(1) \subseteq G_n(1)$. Conversely, the size of both $\left( m+3G_{\frac{n}{3}}(1) \right) \cup \left( 2m+3G_{\frac{n}{3}}(1) \right)$ and $G_n(1)$ are same, and hence both are equal.
\item Assume that $t=1$ and $m \equiv 1 \Mod 3$. Using Part (i), we have $m+3G_{\frac{n}{3}}(1) \subseteq G_{n,3}^1(1)$, and each element of $m+3G_{\frac{n}{3}}(1)$ is congruent to $1$ modulo $3$. Hence $G_{n,3}^1(1)=m+3G_{\frac{n}{3}}(1)$. Similarly, if $t=1$ and $m \equiv 2 \Mod 3$ then $G_{n,3}^1(1)=2m+3G_{\frac{n}{3}}(1)$.
\item The proof is similar to Part (i).
\item The proof is similar to Part (i).
\item Use Part (iii) and Part (iv).
\end{enumerate}
\end{proof}

Let $\Im(z)$ denote the imaginary part of the complex number $z$. 

\begin{lema}\label{Tn(q)SumEqualTo3TimesSum00}
Let $n=3m$ with $m\not\equiv 0 \Mod 3$. Then 
$$T_n(q)= \left\{ \begin{array}{rl}
					-2\sqrt{3} \Im(\omega_3^{q})C_{\frac{n}{3}}(q) & \mbox{if } m \equiv 1 \Mod 3 \\
					-2\sqrt{3} \Im(\omega_3^{2q})C_{\frac{n}{3}}(q) & \mbox{if } m \equiv 2 \Mod 3.
				\end{array}\right.$$  Moreover,  $\frac{T_n(q)}{3}$ is an integer for all $q\in \mathbb{Z}$.
\end{lema}
\begin{proof} We have
		\begin{equation*}
			\begin{split}
				T_{n}(q) &=  \sum\limits_{a\in G_{n,3}^1(1)} i\sqrt{3} (\omega_n^{aq} - \omega_n^{-aq}) \\
				&= \left\{ \begin{array}{rl}
					\sum\limits_{a\in G_{\frac{n}{3}}(1)} i\sqrt{3}(\omega_n^{mq} \omega_n^{3aq} - \omega_n^{-mq} \omega_n^{-3aq}) & \mbox{if } m \equiv 1 \Mod 3 \\
					\sum\limits_{a\in G_{\frac{n}{3}}(1)} i\sqrt{3} (\omega_n^{2mq}\omega_n^{3aq} - \omega_n^{-2mq} \omega_n^{-3aq}) & \mbox{if } m \equiv 2 \Mod 3
				\end{array}\right.\\
				&= \left\{ \begin{array}{rl}
					-2\sqrt{3} \Im(\omega_n^{mq})\sum\limits_{a\in G_{\frac{n}{3}}(1)}  \omega_{\frac{n}{3}}^{aq}  & \mbox{if } m \equiv 1 \Mod 3 \\
					-2\sqrt{3} \Im(\omega_n^{2mq})\sum\limits_{a\in G_{\frac{n}{3}}(1)}  \omega_{\frac{n}{3}}^{aq} & \mbox{if } m \equiv 2 \Mod 3
				\end{array}\right.\\
				&= \left\{ \begin{array}{rl}
					-2\sqrt{3} \Im(\omega_3^{q})C_{\frac{n}{3}}(q) & \mbox{if } m \equiv 1 \Mod 3 \\
					-2\sqrt{3} \Im(\omega_3^{2q})C_{\frac{n}{3}}(q) & \mbox{if } m \equiv 2 \Mod 3.
				\end{array}\right.
			\end{split} 
		\end{equation*} Here the second equality follows from Part $(ii)$ of Lemma~\ref{NewLemmaEquivaTrans1}. Since $2\Im(\omega_3)=\sqrt{3}$, therefore $\frac{T_n(q)}{3}$ is an integer for all $q\in \mathbb{Z}$.
\end{proof}

\begin{lema}\label{Tn(q)SumEqualTo3TimesSum}
Let $n=3^tm$ with $m\not\equiv 0 \Mod 3$ and $t\geq 2$. Then 
$$2T_n(q)= \left\{ \begin{array}{rl}
					-2\sqrt{3} \Im(\omega_n^{mq}+\omega_n^{4mq}+\omega_n^{7mq})C_{\frac{n}{3}}(q) & \mbox{if } m \equiv 1 \Mod 3 \\
					-2\sqrt{3} \Im(\omega_n^{2mq}+\omega_n^{5mq}+\omega_n^{8mq})C_{\frac{n}{3}}(q) & \mbox{if } m \equiv 2 \Mod 3.
				\end{array}\right.$$ Moreover,  $\frac{T_n(q)}{3}$ is an integer for all $q\in \mathbb{Z}$.
\end{lema}
\begin{proof} We use the fact that $G_{n,3}^1(1)$ can be written as disjoint unions in two different ways using Part $(iii)$ and Part $(iv)$ of Lemma~\ref{NewLemmaEquivaTrans1}. We have
		\begin{equation*}
			\begin{split}
			  & ~~	2T_{n}(q) \\
				=&   \sum\limits_{a\in G_{n,3}^1(1)} i\sqrt{3} (\omega_n^{aq} - \omega_n^{-aq}) + \sum\limits_{a\in G_{n,3}^1(1)} i\sqrt{3} (\omega_n^{aq} - \omega_n^{-aq}) \\
				= & \left\{ \begin{array}{rl}
					\sum\limits_{a\in G_{\frac{n}{3}}(1)} i\sqrt{3}((\omega_n^{mq}+\omega_n^{4mq}+\omega_n^{7mq}) \omega_n^{3aq} - (\omega_n^{-mq}+\omega_n^{-4mq}+\omega_n^{-7mq}) \omega_n^{-3aq}) & \mbox{if } m \equiv 1 \Mod 3 \\
					\sum\limits_{a\in G_{\frac{n}{3}}(1)} i\sqrt{3} ((\omega_n^{2mq}+\omega_n^{5mq}+\omega_n^{8mq})\omega_n^{3aq} - (\omega_n^{-2mq}+\omega_n^{-5mq}+\omega_n^{-8mq}) \omega_n^{-3aq}) & \mbox{if } m \equiv 2 \Mod 3
				\end{array}\right.\\
				= & \left\{ \begin{array}{rl}
					 -2\sqrt{3} \Im(\omega_n^{mq}+\omega_n^{4mq}+\omega_n^{7mq})\sum\limits_{a\in G_{\frac{n}{3}}(1)}  \omega_{\frac{n}{3}}^{aq}  & \mbox{if } m \equiv 1 \Mod 3 \\
					-2\sqrt{3} \Im(\omega_n^{2mq}+\omega_n^{5mq}+\omega_n^{8mq})\sum\limits_{a\in G_{\frac{n}{3}}(1)}  \omega_{\frac{n}{3}}^{aq} & \mbox{if } m \equiv 2 \Mod 3
				\end{array}\right.\\
			=	& \left\{ \begin{array}{rl}
					-2\sqrt{3} \Im(\omega_n^{mq}+\omega_n^{4mq}+\omega_n^{7mq})C_{\frac{n}{3}}(q) & \mbox{if } m \equiv 1 \Mod 3 \\
					-2\sqrt{3} \Im(\omega_n^{2mq}+\omega_n^{5mq}+\omega_n^{8mq})C_{\frac{n}{3}}(q) & \mbox{if } m \equiv 2 \Mod 3.
				\end{array}\right.
			\end{split} 
		\end{equation*} 
Here the second equality follows from Part $(v)$ of Lemma~\ref{NewLemmaEquivaTrans1}. If $t=2$, then
		\begin{equation*}
		\begin{split}
			2T_{n}(q)= \left\{ \begin{array}{rl}
				C_9(q)C_{\frac{n}{3}}(q) & \mbox{if } m \equiv 1 \Mod 3 \\
				C_9(q)C_{\frac{n}{3}}(q) & \mbox{if } m \equiv 2 \Mod 3.
			\end{array}\right.
		\end{split} 
	\end{equation*} Thus  $\frac{T_n(q)}{3}$ is an integer for all $q\in \mathbb{Z}$. Assume that $t \geq 3$. If $C_{\frac{n}{3}}(q)\neq 0$ then both $-2\sqrt{3} \Im(\omega_n^{mq}+\omega_n^{4mq}+\omega_n^{7mq})$ and $-2\sqrt{3} \Im(\omega_n^{2mq}+\omega_n^{5mq}+\omega_n^{8mq})$ are rational algebraic integers, and hence both are integers for each $q\in \mathbb{Z}$. As $C_{\frac{n}{3}}(q)$ is an integer multiple of $3$, we find that $\frac{T_n(q)}{3}$ is an integer for all $q\in \mathbb{Z}$.
\end{proof}

\begin{lema}\label{SameParitySum}
Let $n \equiv 0 \Mod {3}$. Then $C_n(q)$ and $\frac{T_n(q)}{3}$ are integers of the same parity for all $q\in \mathbb{Z}$.
\end{lema}
\begin{proof}
Since $T_n(q)- C_n(q)=2 Z_n^1(q)$ is an even integer,  $C_n(q)$ and $T_n(q)$ are integers of the same parity for each $q\in \mathbb{Z}$. By Lemma~\ref{Tn(q)SumEqualTo3TimesSum00} and Lemma~\ref{Tn(q)SumEqualTo3TimesSum}, $\frac{T_n(q)}{3}$ is an integer for all $q\in \mathbb{Z}$. Hence $C_n(q)$ and $\frac{T_n(q)}{3}$ are integers of the same parity for all $q\in \mathbb{Z}$.
\end{proof}	

Let $S$ be a subset of $\mathbb{Z}_n$ and $j\in \{0,1,...,n-1\}$. Define 
$$\alpha_j(S)=\sum\limits_{k \in S \setminus \overline{S}} \omega_n^{jk}\hspace{0.5cm} \textnormal{ and }\hspace{0.5cm} \beta_j(S)=\sum\limits_{k \in \overline{S}}(\omega \omega_n^{jk} + \overline{\omega} \omega_n^{-jk}),$$  
where $\omega = \frac{1}{2} - \frac{i\sqrt{3}}{6}$. 
It is clear that $\alpha_j(S)$ and $\beta_j(S)$ are real numbers. We have 
$$\sum_{k \in S} \omega_n^{jk}=\alpha_j(S)+\beta_j(S)+ \bigg( \frac{-1}{2} + \frac{i\sqrt{3}}{2} \bigg) (\beta_j(S)-\beta_{n-j}(S)).$$ 
Note that $\alpha_j(S)=\alpha_{n-j}(S)$ for each $j$. Therefore if $\alpha_j(S)+\beta_j(S)\in \mathbb{Z}$ for each $j\in \{0,1,...,n-1\}$ then $\beta_j(S)-\beta_{n-j}(S)=[\alpha_j(S)+\beta_j(S)]-[\alpha_{n-j}(S)+\beta_{n-j}(S)]$ is also an integer for each $j\in \{0,1,...,n-1\}$. Hence the mixed circulant graph $\text{Circ}(\mathbb{Z}_n,S)$ is Eisenstein integral if and only if $\alpha_j(S)+\beta_j(S)$ is an integer for each $j\in \{0,1,...,n-1\}$.

\begin{lema}\label{CharaEisensteinIntegral}
Let $S\subseteq \mathbb{Z}_n$ such that $0\notin S$. Then the mixed circulant graph $\text{\text{Circ}}(\mathbb{Z}_n,S)$ is Eisenstein integral if and only if $2 \alpha_j(S)$ and $2 \beta_j(S)$ are integers of the same parity for each $j\in \{0,1,...,n-1\}$.
\end{lema}
\begin{proof} 
Suppose the mixed circulant graph $\text{\text{Circ}}(\mathbb{Z}_n,S)$ is Eisenstein integral and $j\in \{0,1,...,n-1\}$. Then $\alpha_j(S)+\beta_j(S)$ and $\beta_j(S)-\beta_{n-j}(S)= \sum\limits_{k\in \overline{S}} \frac{-i\sqrt{3}}{3}(\omega_n^{jk}-\omega_n^{-jk})$ are integers. By Lemma~\ref{Sqrt3NecessIntSum}, $\sum\limits_{k\in \overline{S}\cup \overline{S}^{-1}} \omega_n^{jk}\in \Zl$. Since 
$$2 \beta_j(S)= \sum\limits_{k\in \overline{S}\cup \overline{S}^{-1}} \omega_n^{jk}-  \sum\limits_{k\in \overline{S}}\frac{i\sqrt{3}}{3} (\omega_n^{jk}-\omega_n^{-jk}),$$ 
we find that $2 \beta_j(S)$ is an integer. Therefore $2 \alpha_j(S)=2(\alpha_j(S)+\beta_j(S))-2\beta_j(S)$ is also an integer of the same parity with $2\beta_j(S)$.

Conversely, assume that $2 \alpha_j(S)$ and $2 \beta_j(S)$ are integers of the same parity for each $j\in \{0,1,...,n-1\}$. Then $\alpha_j(S)+\beta_j(S)$ is an integer for each $j\in \{0,1,...,n-1\}$. Hence the mixed circulant graph $\text{\text{Circ}}(\mathbb{Z}_n,S)$ is Eisenstein integral.
\end{proof}

\begin{lema}\label{CharaEisensteinIntegral2}
Let $S\subseteq \mathbb{Z}_n$ such that $0\notin S$. Then the mixed circulant graph $\text{\text{Circ}}(\mathbb{Z}_n,S)$ is Eisenstein integral if and only if $\alpha_j(S)$ and $\beta_j(S)$ are integers for each $j\in \{0,1,...,n-1\}$.
\end{lema}
\begin{proof}
By Lemma~\ref{CharaEisensteinIntegral}, it is enough to show that $2 \alpha_j(S)$ and $2 \beta_j(S)$ are integers of the same parity if and only if $\alpha_j(S)$ and $\beta_j(S)$ are integers. If $\alpha_j(S)$ and $\beta_j(S)$ are integers, then clearly $2 \alpha_j(S)$ and $2 \beta_j(S)$ are even integers. Conversely, assume that $2 \alpha_j(S)$ and $2 \beta_j(S)$ are integers of the same parity. Since $\alpha_j(S)$ is an algebraic integer, the integrality of $2 \alpha_j(S)$ implies that $\alpha_j(S)$ is an integer. Thus $2 \alpha_j(S)$ is even, and so by the assumption $2 \beta_j(S)$ is also an even integer. Hence $\beta_j(S)$ is an integer.
\end{proof}

\begin{theorem}\label{MinCharacEisensteinInteg}
Let $S\subseteq \mathbb{Z}_n$ such that $0\notin S$. Then the mixed circulant graph $\text{\text{Circ}}(\mathbb{Z}_n,S)$ is Eisenstein integral if and only if $\text{\text{Circ}}(\mathbb{Z}_n,S)$ is HS-integral.
\end{theorem}

\begin{proof} 
By Lemma~\ref{CharaEisensteinIntegral2}, it is enough to show that $ \alpha_j(S)$ and $\beta_j(S)$ are integers for each $j\in\{0,1,...,n-1\}$ if and only if $\text{\text{Circ}}(\mathbb{Z}_n,S)$ is HS-integral. Note that $\alpha_j(S)$ is an eigenvalue of the circulant graph $\text{\text{Circ}}(\mathbb{Z}_n,S\setminus \overline{S})$. By Theorem~\ref{2006integral}, $\alpha_j(S)$ is an integer for each $j\in\{0,1,...,n-1\}$ if and only if $S \setminus \overline{S} = \bigcup\limits_{d\in \mathscr{D}_1}G_n(d)$ for some $\mathscr{D}_1 \subseteq \{ d: d\mid n\}$. Assume that $\alpha_j(S)$ and $\beta_j(S)$ are integers for each $j$. Then $-\frac{i\sqrt{3}}{3}\sum\limits_{k \in \overline{S}} (\omega_n^{jk}- \omega_n^{-jk})=\beta_j(S)-\beta_{n-j}(S)$ is also an integer for each $j$. Using Theorem~\ref{2006integral} and Lemma~\ref{CharaNewIntegSum}, we see that $S \setminus \overline{S}$ and $\overline{S}$ satisfy the conditions of Theorem~\ref{CharaOfIntegMixedGraph}. Hence $\text{\text{Circ}}(\mathbb{Z}_n,S)$ is HS-integral.

Conversely, assume that $\text{\text{Circ}}(\mathbb{Z}_n,S)$ is HS-integral. Then $\text{\text{Circ}}(\mathbb{Z}_n,S\setminus \overline{S})$ is integral, and hence $\alpha_j(S)$ is an integer for each $j\in\{0,1,...,n-1\}$. By Theorem~\ref{CharaOfIntegMixedGraph}, we have
\begin{equation*}
			\overline{S}= \left\{ \begin{array}{ll}
				\emptyset & \text{ if } n\not\equiv 0\Mod 3 \\ 
				\bigcup\limits_{d\in \mathscr{D}_2}S_n(d) & \text{ if } n\equiv 0\Mod 3,
			\end{array}\right.
		\end{equation*}
where $\mathscr{D}_2 \subseteq \{ d: d\mid \frac{n}{3}\}$ and $S_n(d)\in \{ G_{n,3}^{1}(d) , G_{n,3}^{2}(d)\}$. Then 
\begin{equation*}
\begin{split}
\beta_j(S) &= \frac{1}{2} \sum\limits_{k \in \overline{S} \cup \overline{S}^{-1}} \omega_n^{jk} - \frac{1}{6}  \sum\limits_{k \in \overline{S}} i\sqrt{3} (\omega_n^{jk} - \omega_n^{-jk})\\
&= \frac{1}{2} \sum\limits_{d\in \mathcal{D}_2}\sum\limits_{k \in G_n(d)} \omega_n^{jk} + \frac{1}{6} \sum\limits_{d\in \mathcal{D}_2}\sum\limits_{k \in S_n(d)} i\sqrt{3} (\omega_n^{jk} - \omega_n^{-jk})\\
&= \frac{1}{2} \sum\limits_{d\in \mathcal{D}_2}C_{\frac{n}{d}}(j) \pm \frac{1}{6} \sum\limits_{d\in \mathcal{D}_2}T_{\frac{n}{d}}(j)\\
&= \frac{1}{2} \sum\limits_{d\in \mathcal{D}_2} \left( C_{\frac{n}{d}}(j) \pm \frac{1}{3} T_{\frac{n}{d}}(j)\right).
\end{split} 
\end{equation*}
By Lemma~\ref{SameParitySum}, $C_{\frac{n}{d}}(j) \pm \frac{1}{3} T_{\frac{n}{d}}(j)$ are even integers for all $d\in \mathcal{D}_2$. Hence $\beta_j(S)$ is an integer for each $j\in\{0,1,...,n-1\}$.
\end{proof}
The following example illustrates Theorem~\ref{MinCharacEisensteinInteg}.
\begin{ex}
Consider the HS-integral graph $\text{\text{Circ}}(\mathbb{Z}_{12}, S)$ of Example~\ref{ex2}. By Theorem~\ref{MinCharacEisensteinInteg}, the graph $\text{\text{Circ}}(\mathbb{Z}_{12}, S)$ is Eisenstein integral. Indeed, the eigenvalues of  $\text{\text{Circ}}(\mathbb{Z}_{12}, S)$ are obtained as
\begin{equation*}
			\begin{split}
\gamma_j=&2\cos\left(\frac{\pi j}{3}\right) + \cos\left(\frac{5\pi j}{6}\right) + \cos\left(\frac{11\pi j}{6}\right) + \frac{1}{\sqrt{3}} \left[ \sin\left(\frac{5\pi j}{6}\right) + \sin\left(\frac{11\pi j}{6}\right) \right]\\
&+ \omega_3 \frac{2}{\sqrt{3}}\left[ \sin\left(\frac{5\pi j}{6}\right) + \sin\left(\frac{11\pi j}{6}\right) \right]~~\text{for each } j \in \mathbb{Z}_{12}.
			\end{split}\end{equation*}
One can see that $\gamma_0=4, \gamma_1=1, \gamma_2=-1-2\omega_3, \gamma_3=-2, \gamma_4=-3-2\omega_3, \gamma_5=1, \gamma_6=0, \gamma_7=1,$ $ \gamma_8=-1+2\omega_3, \gamma_9=-2, \gamma_{10}=1+2\omega_3$ and $\gamma_{11}=1$. Thus $\gamma_j$ is an Eisenstein integer for each $j \in \mathbb{Z}_{12}$. \qed
\end{ex}

	\section{Eigenvalues and HS-eigenvalues of unitary oriented circulant graphs in terms of generalized M$\ddot{\text{o}}$bius function}\label{airthmetic}

	Let $n \equiv 0 \Mod 3$. The underlying graph of $\text{\text{Circ}}(\mathbb{Z}_n,G_{n,3}^1(1))$ is known as an unitary simple circulant graph. The graph $\text{\text{Circ}}(\mathbb{Z}_n,G_{n,3}^1(1))$ is called an \textit{unitary oriented circulant graph} (\textit{UOCG}). Using Theorem~\ref{CharaOfIntegMixedGraph} and Theorem~\ref{MinCharacEisensteinInteg}, UOCG is an HS-integral as well as Eisenstein integral graph. The eigenvalues and the  HS-eigenvalues of UOCG are $\frac{C_n(j)}{2} + \frac{ T_n(j)}{6}+ \omega_3^2 \frac{T_n(j)}{3}$ and $\frac{C_n(j)}{2} + \frac{ T_n(j)}{2}$, respectively, for each $j \in \{0,1,\ldots, n-1\}$. Note that $C_n(j)$ is an eigenvalue of the underlying graph of UOCG for each $j \in \{0,1,\ldots, n-1\}$. Using Lemma~\ref{Tn(q)SumEqualTo3TimesSum00} and Lemma~\ref{Tn(q)SumEqualTo3TimesSum}, we can express $T_n(j)$ in terms of $C_{\frac{n}{3}}(j)$. This, in turn, express the eigenvalues and the HS-eigenvalues of UOCG in terms of Ramanujan sums. The Ramanujan sum $C_n(j)$ is well known~\cite{murty2008problems} in terms of M$\ddot{\text{o}}$bius function. In particular, we have
	\[C_n(j)=\sum\limits_{d \mid \delta} d \mu(n/d)=\frac{\mu(n/\delta) \varphi(n)}{\varphi(n/\delta)},\]
where $\delta=\gcd(n,j)$. We attempt to obtain similar expression for $T_n(j)$ in terms of generalized M$\ddot{\text{o}}$bius function. That, in turn, will express the eigenvalues and the HS-eigenvalues of UOCG in terms of generalized M$\ddot{\text{o}}$bius function.

	The classical M$\ddot{\text{o}}$bius function $\mu(n)$ is defined by 
	$$\mu(n)= \left\{ \begin{array}{cl}
		1& \mbox{if } n=1 \\ 
		(-1)^k & \mbox{if $n$ is a product of $k$ distinct primes} \\
		0   &\mbox{otherwise.} 
	\end{array}\right. $$
	
	Let $\delta$ be the indicator function defined by
	$$\delta(n)= \left\{ \begin{array}{lll}
		1& \mbox{if}
		& n=1 \\ 0 & \mbox{if} & n>1 . 
	\end{array}\right. $$
	
	\begin{theorem} \cite{murty2008problems}~~~ $ \sum\limits_{d\mid n} \mu(d) =\delta(n).$
	\end{theorem}
	
	E. Cohen \cite{cohen} introduced a generalized M$\ddot{\text{o}}$bius inversion formula of arbitrary direct factor sets. Let $P$ and $Q$ be two non-empty subsets of $\mathbb{N}$ such that if $n_1,n_2 \in \mathbb{N}$ with $\gcd(n_1,n_2) = 1$, then $n_1n_2 \in P$ (resp. $n_1n_2 \in Q$) if and only if $ n_1,n_2 \in P$ (resp. $ n_1,n_2 \in Q$). If each integer $n \in \mathbb{N}$ possesses a unique factorization of the form $n = ab$ with $a \in P, b \in Q$, then the sets $P$ and $Q$ are called direct factor sets of $\mathbb{N}$. In what follows, $P$ will denote such a direct factor set with (conjugate) factor set $Q$.
	The M$\ddot{\text{o}}$bius function can be generalized to an arbitrary direct factor set $P$ by setting
 $$\mu_P(n)= \sum_{d\mid n,d\in P} \mu\bigg(\frac{n}{d}\bigg),$$ where $\mu$ is the classical M$\ddot{\text{o}}$bius function. For example, $\mu(n) = \mu_{\{1\}}(n)$ and $\mu_{\mathbb{N}}(n) =\delta(n)$. 
	
	\begin{theorem}\cite{cohen} ~~~$ \sum\limits_{d\mid n, d\in Q} \mu_P\left( \frac{n}{d} \right)=\delta(n).$  
	\end{theorem}
	
	\begin{theorem}\cite{cohen}\label{cohenMobiInvForm} If $f(n)$ and $g(n)$ are arithmetic functions then 
		$$f(n)= \sum_{d\mid n, d\in Q} g\left(\frac{n}{d} \right) \textnormal{ if and only if } g(n)= \sum_{d\mid n} f(d) \mu_P\left( \frac{n}{d}\right).$$
	\end{theorem}
	
	For the remaining part of this section, we consider the direct factors $P=\{2^k: k\geq 0 \}$, and $Q$, the set of all odd natural numbers.
	
	\begin{lema} Let $P=\{2^k: k\geq 0 \}$. Then
		$$\mu_P(n)= \left\{ \begin{array}{lll}
		0& \mbox{if $n$ is even}
		\\ \mu(n) & \mbox{if $n$ is odd.} 
	\end{array}\right. $$
	\end{lema}

Note that if $n\equiv 0 \Mod {3}$ then $D_{n,3}^2 = D_{\frac{n}{3},3}^2$.
	
		\begin{theorem}\label{sumformulaThird} 
		Let $n=3m$  with $m$ an odd integer and $D_{n,3}^2=\emptyset$. Then 
		\begin{align*}
		T_n(j)=\sum_{\substack{d \mid \frac{n}{3} \\ \frac{j}{d}\equiv 1 \Mod {3} }} -3d \mu \left(\frac{n}{3d}\right) + \sum_{\substack{d \mid \frac{n}{3} \\ \frac{j}{d}\equiv 2 \Mod {3} }} 3d \mu \left(\frac{n}{3d}\right).
		\end{align*}
	\end{theorem}
	\begin{proof}
		Let \begin{equation*}
			\begin{split}
				f_n(j) =\sum_{a\in M_{n,3}^1(1) } i\sqrt{3}(\omega_n^{aj}- \omega_n^{-aj}) 
				 = \left\{ \begin{array}{rl}
			-n& \mbox{ if }  j\equiv \frac{n}{3} \Mod n \\
			n& \mbox{ if }  j\equiv \frac{2n}{3} \Mod n \\
			0 &\mbox{ otherwise.} 
		\end{array}\right.
			\end{split} 
		\end{equation*}
		Using Lemma~\ref{SecLemmaSetEqua}, we have
		\begin{equation*}
			\begin{split}
				f_n(j)= \sum_{a\in M_{n,3}^1(1) } i\sqrt{3}(\omega_n^{aj}- \omega_n^{-aj}) &=  \sum_{d\in D_{n,3}^1} \sum_{a\in G_{n,3}^1(d) } i\sqrt{3}(\omega_n^{aj}- \omega_n^{-aj})\\ 
				&=  \sum_{d\in D_{n,3}^1} \sum_{a\in dG_{\frac{n}{d},3}^1(1) } i\sqrt{3}(\omega_n^{aj}- \omega_n^{-aj})\\
				&= \sum_{d\in D_{n,3}^1} \sum_{a\in G_{\frac{n}{d},3}^1(1) } i\sqrt{3}[(\omega_n^d)^{aj}- (\omega_n^d)^{-aj}] 
				= \sum_{d\in D_{n,3}^1} T_{\frac{n}{d}}(j).
			\end{split} 
		\end{equation*} 
In the last equation, we have used the fact that $\omega_n^d=\exp(\frac{2\pi i}{n/d})$ is a primitive $\frac{n}{d}$-th root of unity. Since $D_{n,3}^1 \subseteq Q$, from Theorem~\ref{cohenMobiInvForm} we have

		\begin{align*}
				T_n(j)=&\sum_{\substack{d \mid n \\ j\equiv \frac{d}{3} \Mod d }} f_d(j) \mu_P\left(\frac{n}{d}\right) + \sum_{\substack{d \mid n \\ j\equiv \frac{2d}{3} \Mod d }} f_d(j) \mu_P\left(\frac{n}{d}\right) \\
				=&\sum_{\substack{d \mid n \\ j\equiv \frac{d}{3} \Mod d }} -d \mu_P\left(\frac{n}{d}\right) + \sum_{\substack{d \mid n \\ j\equiv \frac{2d}{3} \Mod d }} d \mu_P\left(\frac{n}{d}\right)  \\
				=&\sum_{\substack{3d \mid n \\ j\equiv d \Mod {3d} }} -3d \mu\left(\frac{n}{3d}\right) + \sum_{\substack{3d \mid n \\ j\equiv 2d \Mod {3d} }} 3d \mu\left(\frac{n}{3d}\right) \\
				=&\sum_{\substack{d \mid \frac{n}{3} \\ \frac{j}{d}\equiv 1 \Mod {3} }} -3d \mu\left(\frac{n}{3d}\right) + \sum_{\substack{d \mid \frac{n}{3} \\ \frac{j}{d}\equiv 2 \Mod {3} }} 3d \mu\left(\frac{n}{3d}\right).
		\end{align*} 		
	\end{proof}

The case that $n=3m$ with $m$ an even integer is not covered in Theorem \ref{sumformulaThird}. Now assume that $n=3m$ with $m$ an even integer, so that $n\equiv 0 \Mod 6$.  For a divisor $d$ of $\frac{n}{6},r\in \{1,5\}$ and $g\in \mathbb{Z}$, define the following sets:
\begin{align*}
& M_{n,6}^r(d) =  \{dk: 0\leq dk < n , k \equiv r \Mod 6 \};\\
& G_{n,6}^r(d)= \{ dk: 1\leq dk < n ,k\equiv r \Mod 6, \gcd(dk,n )= d \};~\text{ and} \\
& D_{g,6}^r = \{ k: k \text{ divides } g, k \equiv r \Mod 6\} .
\end{align*}

	\begin{lema}\label{SecLemmaSetEquaForSix} Let $n\equiv 0 \Mod 6$, $d$ divides $\frac{n}{6}$ and $g= \frac{n}{6d}$. Then the following hold:
		\begin{enumerate}[label=(\roman*)]
			\item $G_{n,6}^1(d) \cap G_{n,6}^5(d)=\emptyset$;
			\item $G_n(d)=G_{n,6}^1(d) \cup G_{n,6}^5(d)$;
			\item $M_{n,6}^1(d) =\bigg( \bigcup\limits_{h\in D_{g,6}^1} G_{n,6}^1(hd) \bigg) \cup \bigg( \bigcup\limits_{h\in D_{g,6}^5} G_{n,6}^5(hd) \bigg)$; 
			\item $M_{n,6}^5(d) =\bigg( \bigcup\limits_{h\in D_{g,6}^1} G_{n,6}^5(hd) \bigg) \cup \bigg( \bigcup\limits_{h\in D_{g,6}^5} G_{n,6}^1(hd) \bigg)$.
		\end{enumerate}
	\end{lema}
	\begin{proof}
	The proof is similar to the proof of Lemma $\ref{SecLemmaSetEqua}$.
	\end{proof}
	
Note that $D_{n,6}^5 = D_{\frac{n}{6},6}^5$. In the next result, we calculate $T_n(j)$ for the values of $n$ not covered in Theorem \ref{sumformulaThird}.
	
	\begin{theorem}\label{sumformula} 
		Let $n\equiv 0 \Mod 6$ and $D_{n,6}^5=\emptyset$. Then 
		\begin{align*}
		T_n(j)=\sum_{\substack{d \mid \frac{n}{6} \\ \frac{j}{d}\equiv 1 \mbox{ or } 2 \Mod {6} }} -3d \mu_P\left(\frac{n}{6d}\right) +  \sum_{\substack{d \mid \frac{n}{6} \\ \frac{j}{d}\equiv 4 \mbox{ or } 5 \Mod {6} }} 3d \mu_P\left(\frac{n}{6d}\right),
		\end{align*}
	\end{theorem}
	\begin{proof}
		Let \begin{equation*}
			\begin{split}
				f_n(j) =\sum_{a\in M_{n,6}^1(1) } i\sqrt{3}(\omega_n^{aj}- \omega_n^{-aj})
				= \left\{ \begin{array}{rl}
			-\frac{n}{2} & \mbox{ if }  j\equiv \frac{n}{6} \Mod n \\
			-\frac{n}{2} & \mbox{ if }  j\equiv \frac{2n}{6} \Mod n \\
			0 & \mbox{ if }  j\equiv \frac{3n}{6} \Mod n \\
			\frac{n}{2} & \mbox{ if }  j\equiv \frac{4n}{6} \Mod n \\
			\frac{n}{2} & \mbox{ if }  j\equiv \frac{5n}{6} \Mod n \\
			0 &\mbox{ otherwise.} 
		\end{array}\right.
			\end{split} 
		\end{equation*}
By Lemma~\ref{SecLemmaSetEquaForSix} we get
		\begin{equation*}
			\begin{split}
				f_n(j)= \sum_{a\in M_{n,6}^1(1) } i\sqrt{3}(\omega_n^{aj}- \omega_n^{-aj}) &=  \sum_{d\in D_{n,6}^1} \sum_{a\in G_{n,6}^1(d) }  i\sqrt{3}(\omega_n^{aj}- \omega_n^{-aj}) \\ 
				&=  \sum_{d\in D_{n,6}^1} \sum_{a\in dG_{\frac{n}{d},6}^1(1) }  i\sqrt{3}(\omega_n^{aj}- \omega_n^{-aj}) \\
				&= \sum_{d\in D_{n,6}^1} \sum_{a\in G_{\frac{n}{d},6}^1(1) }  i\sqrt{3}[(\omega_n^d)^{aj}- (\omega_n^d)^{-aj}] \\
				&= \sum_{d\in D_{n,6}^1} T_{\frac{n}{d}}(j).
			\end{split} 
		\end{equation*}
In the last equation, we have used the fact that $\omega_n^d$ is a primitive $\frac{n}{d}$-th root of unity and $G_{\frac{n}{d},6}^1(1)= G_{\frac{n}{d},3}^1(1)$ for all $d\in D_{n,6}^1$. Since $D_{n,6}^1 \subset Q$, by Theorem~\ref{cohenMobiInvForm} we get

		\begin{align*}
				T_n(j)=&\sum_{\substack{d \mid n \\ j\equiv \frac{d}{6} \Mod d }} f_d(j) \mu_P\left(\frac{n}{d}\right) + \sum_{\substack{d \mid n \\ j\equiv \frac{2d}{6} \Mod d }} f_d(j) \mu_P\left(\frac{n}{d}\right)\\
				& + \sum_{\substack{d \mid n \\ j\equiv \frac{4d}{6} \Mod d }} f_d(j) \mu_P\left(\frac{n}{d}\right) + \sum_{\substack{d \mid n \\ j\equiv \frac{5d}{6} \Mod d }} f_d(j) \mu_P\left(\frac{n}{d}\right)\\
				=&\sum_{\substack{d \mid n \\ j\equiv \frac{d}{6} \mbox{ or } \frac{2d}{6} \Mod d }} -\frac{d}{2} \mu_P\left(\frac{n}{d}\right) + \sum_{\substack{d \mid n \\ j\equiv \frac{4d}{6} \mbox{ or } \frac{5d}{6} \Mod d }} \frac{d}{2} \mu_P\left(\frac{n}{d}\right) \\
				=&\sum_{\substack{6d \mid n \\ j\equiv d \mbox{ or } 2d \Mod {6d} }} -3d \mu_P\left(\frac{n}{6d}\right)  + \sum_{\substack{6d \mid n \\ j\equiv 4d \mbox{ or } 5d \Mod {6d} }} 3d \mu_P\left(\frac{n}{6d}\right)\\
				=&\sum_{\substack{d \mid \frac{n}{6} \\ \frac{j}{d}\equiv 1 \mbox{ or } 2 \Mod {6} }}- 3d \mu_P\left(\frac{n}{6d}\right) +  \sum_{\substack{d \mid \frac{n}{6} \\ \frac{j}{d}\equiv 4 \mbox{ or } 5 \Mod {6} }} 3d \mu_P\left(\frac{n}{6d}\right).
		\end{align*} 
	\end{proof}


\end{document}